\documentclass[11pt]{article}
\usepackage{amssymb}
\usepackage{amsmath}
\usepackage{mathrsfs}
\usepackage{graphics}
\usepackage{graphicx}
\usepackage{xcolor}
\usepackage{subfigure}
\usepackage[T1]{fontenc}
\usepackage{latexsym,amssymb,amsmath,amsfonts,amsthm}\usepackage{txfonts}
\newcommand{\be}{\begin{eqnarray}}
\newcommand{\ben}{\begin{eqnarray*}}
\newcommand{\en}{\end{eqnarray}}
\newcommand{\enn}{\end{eqnarray*}}

\newtheorem{theorem}{Theorem}[section]
\newtheorem{lemma}{Lemma}[section]
\newtheorem{prp}[theorem]{Proposition}
\newtheorem{thm}[theorem]{Theorem}
\newtheorem{cor}[theorem]{Corollary}
\newtheorem{dfn}{Definition}[section]

\newtheorem{remark}{Remark}

\begin{document}
\renewcommand{\theequation}{\arabic{section}.\arabic{equation}}
\begin{titlepage}
\title{\bf Splitting up method for 2D stochastic primitive equations with multiplicative noise
}
\author{ Xuhui Peng$^{1}$,\ \ Rangrang Zhang$^{2,}$\thanks{Corresponding author.} \\
{\small $^1$
LCSM, Ministry of Education, School of Mathematics and Statistics, }\\
{\small Hunan Normal University, Changsha, Hunan 410081, P. R. China, }\\
{\small $^2$  Department of  Mathematics,
Beijing Institute of Technology, Beijing, 100081, China.}\\
 {\sf xhpeng@hunnu.edu.cn},\ {\sf rrzhang@amss.ac.cn}}
\date{}
\end{titlepage}
\maketitle

\noindent\textbf{Abstract}:
 This paper concerns the convergence of an iterative scheme for 2D stochastic primitive equations on a bounded domain. The stochastic system is split into two equations: a deterministic 2D primitive equations with random initial value and a linear stochastic parabolic equation, which are both simpler for numerical computations. An estimate of approximation error is given, which implies that the strong speed rate of the convergence in probability is almost $\frac{1}{2}$.



\section{Introduction}
In this paper, we focus on the convergence of some iterative schemes for 2D stochastic primitive equations, which is helpful for numerical approximation. As a fundamental model in meteorology, the primitive equations
were derived from the Navier-Stokes equations, with rotation, coupled with thermodynamics
and salinity diffusion-transport equations (see \cite{L-T-W-1,L-T-W-2,JP}). This model in the deterministic case has been intensively investigated because of the interests stemmed from physics and mathematics. For example, the mathematical study of the primitive equations originated in a series of articles by Lions, Temam, and Wang in the early 1990s (see \cite{L-T-W-1,L-T-W-2,L-T-W-3,L-T-W-4} and the references therein), where they set up the mathematical framework and showed the global existence of weak solutions. Cao and Titi \cite{C-T-1} developed an approach to dealing with the $L^6$-norm of the fluctuation $\tilde{v}$ of horizontal velocity and obtained the global well-posedness for the 3D viscous primitive equations.
\

Along with the great successful developments of deterministic
primitive equations, the random situation has also been developed rapidly.  For 3D stochastic primitive equations, Guo and Huang \cite{Guo} obtained
the existence of universal random attractor of strong solution under the assumptions that the momentum
equation is driven by an additive stochastic forcing and the
thermodynamical equation is under a fixed heat source. Debussche, Glatt-Holtz, Temam and Ziane \cite{D-G-T-Z} established the global well-posedness of the strong solution when this model is
 driven by multiplicative random noises. Dong et. al. \cite{RR} studied its ergodic theory and proved that all weak solutions which are limits of spectral Galerkin approximations share the same invariant measure. Moreover, they established a large deviation principle for this model in \cite{D-Z-Z}. For 2D stochastic primitive equations, Gao and Sun \cite{G-S} obtained its global well-posedness and Freidlin-Wentzell's large deviations.

The aim of this paper is to study numerical approximations to 2D stochastic primitive equations. There are many literature on this topic for stochastic parabolic differential equations. For example, using the semigroup and the cubature techniques, D\"{o}rsek \cite{D} studied the weak speed of convergence of a certain time-splitting scheme combing with a Galerkin approximation in the space variable for the stochastic Navier-Stokes equations with an additive noise.
The strong convergence of the splitting up method has already been studied in a series of papers by Gy\"{o}ngy and Krylov (see \cite{G-K-1,G-K-2} etc), where the rate of convergence is obtained based on stochastic calculus. However, the linear setting used in their papers does not cover some hydrodynamical models, such as stochastic Navier-Stokes equations, stochastic primitive equations and so on.  Recently, Bessaih, Brze\'{z}niak and Millet \cite{B-B} studied the splitting up method for the strong solution of 2D stochastic Navier-Stokes equations on a torus in the space $L^2([0,T];V)$ and proved that the strong speed of convergence in probability is almost $\frac{1}{2}$.

In this paper, we devote to obtaining the strong speed of the convergence in probability for 2D stochastic primitive equations using the splitting up method from \cite{B-B}. The splitting up method is implemented by using two consecutive steps on each time interval. The first step is to solve the deterministic 2D primitive equations with random initial value. The second step is to solve a stochastic parabolic equation. The corresponding solutions are denoted by  $v^n$ and $\eta^n$ (see (\ref{qq-1}) and (\ref{qq-2})), respectively. Our aim is to establish the approximation error of  $v^n-v$ and $\eta^n-v$ in the space $L^{\infty}([0,T]; H)\cap L^2([0,T]; V)$.
 During the proof process, the uniform $V-$norm estimates $\mathbb{E}\sup_{t\in[0,T]}\|v(t)\|^2$ of strong solution plays a key role (see Proposition \ref{prp-8}). In \cite{B-B}, the authors obtained such estimates of 2D stochastic Navier-Stokes equations by transforming this model into a curvature equation and utilizing its cancelation property in $H\subset L^2$. However, for 2D stochastic primitive equations, we have no uniform $V-$norm estimates, only $\mathbb{E}\int^T_0 \|v(t)\|^2 dt\leq C$ is available, which leads to some difficulties. For example, during the proof process of Proposition \ref{prp-8}, the index of $\|v(t)\|$ in $I(t)$ has to be strictly less than $2$. Otherwise, we will encounter $\mathbb{E}\int^T_0 \|v(t)\|^{\alpha} dt$ for $\alpha>2$ after using H\"{o}lder inequality.
To overcome this difficulty, we
 divide $\|v(t)\|$ into several parts with small index and make use of uniform $V-$estimates of $v^n$ and uniform $H-$estimates of $v, \partial_z v$ (for details, see Proposition \ref{prp-8}). Moreover, in order to obtain the uniform $V-$estimates of $v^n$, an appropriate stopping time is introduced (see Lemma \ref{lem-8}). Besides, compared with 2D stochastic Navier-Stokes equations, we need to make additional $H-$estimates of $\partial_z v^n$ and $\partial_z \eta^n$ appeared in the estimations of nonlinear terms (see Lemma \ref{lemma-1}).

Specifically, for any $n\geq 1$, set the error term
\begin{eqnarray*}
e_n(T)&:=&\sup_{k=0,\cdot\cdot\cdot, n}\Big(|v^n(t^+_{k})-v(t_k)|+|\eta^n(t^-_{k})-v(t_k)|\Big)\\
&&\ +\left(\int^T_0\|v^n(s)-v(s)\|^2ds\right)^{\frac{1}{2}}+\left(\int^T_0\|\eta^n(s)-v(s)\|^2ds\right)^{\frac{1}{2}}.
\end{eqnarray*}
Under some conditions, the main result we obtain is
\begin{thm}\label{thm-4}
Let $\varepsilon \in [0,1)$. Under \textbf{Hypotheses A-C}, the error term $e_n(T)$ converges to 0 in probability with the speed almost $\frac{1}{2}$. Precisely, for any sequence $l(n)_{n\geq 1}$ converging to $\infty$, we have
\begin{eqnarray*}
\lim_{n\rightarrow \infty}\mathbb{P}\Big(e_n(T)\geq \frac{l(n)}{\sqrt{n}}\Big)=0.
\end{eqnarray*}
\end{thm}
Here, $\varepsilon$ is a parameter appeared in Hypotheses A-C, which will be described in Sect. 4 and 5.

This paper is organized as follows: In Sect. 2, the mathematical framework is introduced.  We obtain the global well-posedness of strong solution in Sect. 3.  In Sect. 4, the splitting up method is presented, where two approximation equations to the primitive equations are constructed. Further, $H$ and $V-$norm estimates of the difference between the two approximation equations are established, respectively. Finally, An auxiliary process is introduced for technical reasons.  In Sect. 5, the speed rate of the convergence in probability is obtained.

%

\section{The mathematical framework}
The two dimensional primitive equations can be formally derived from the full three dimensional system under the assumption of invariance with respect to the second horizontal variable $y$ as in \cite{Ziane}.
The 2D primitive equations driven by a stochastic forcing in a Cartesian system can be written as
\begin{eqnarray}\label{eq-1}
\frac{\partial v}{\partial t}-\mu\Delta v+v\partial_{x} v+\theta\partial_{z} v+\partial_{x} p &=&\psi(t,v)\frac{dW}{dt},\\
\label{eq-3}
\partial_{x} v+\partial_{z}\theta&=&0,
\end{eqnarray}
where the velocity $v=v(t,x,z)\in \mathbb{R}$, the vertical velocity $\theta$ and the pressure\ $p$ are all unknown functionals. $(x,z)\in \mathcal{M}=[0,L]\times [-h,0]$. $W$ is a cylindrical Winner process, which will be given in Sect. \ref{st-1}. $\Delta=\partial^{2}_{x}+\partial^{2}_{z}$ is the Laplacian operator. Note that $p$ is independent of the vertical variable $z$.

We impose the following boundary conditions:
\begin{eqnarray}\label{eq1}
 \partial_{z}v=0, \ \theta=0  & &  {\rm {on}} \ \Gamma_{\textit{u}}=(0,L)\times \{0\},\\
 \label{eq2}
 \partial_{z}v=0,\  \theta=0       & &   {\rm {on}}\ \Gamma_{\textit{b}}=(0,L)\times\{-h\},\\
 \label{eq3}
 v=0 & &   {\rm {on}}\  \Gamma_{l}=\{0,L\}\times (-h,0).
 \end{eqnarray}
 Without loss of generality, we assume that
$$
\mu=1, \quad \int^{0}_{-h}v dz=0.
$$
Integrating (\ref{eq-3}) from $-h$ to $z$ and using  (\ref{eq1}), (\ref{eq2}), we have
\begin{equation}
\theta(t,x,z):=\Phi(v)(t,x,z)=-\int^{z}_{-h}\partial_x v (t,x,z')dz'.
\end{equation}
Then, (\ref{eq-1})-( \ref{eq3}) can be rewritten as
 \begin{eqnarray}\label{eq5-1}
&\frac{\partial v}{\partial t}-\Delta v+v\partial_{x} v+\Phi(v)\partial_{z} v+\partial_{x} p =\psi(t,v)\frac{dW}{dt},&\\
\label{eq-10-1}
&\partial_{z}v|_{\Gamma_{\textit{u}}}=0,\quad \partial_{z}v|_{\Gamma_{\textit{b}}}=0,\quad  v|_{\Gamma_{l}}=0.&
\end{eqnarray}
The initial condition is given by
 \begin{equation}\label{eq-11-1}
 v(0)=v_0.
 \end{equation}
\subsection{Some functional spaces }
Let $\mathcal{L}(K_1;K_2)$ (resp. $\mathcal{L}_2(K_1;K_2)$) be the space of bounded (resp. Hilbert-Schmidt) linear operators from the Hilbert space $K_1$ to $K_2$, whose norm is denoted by $\|\cdot\|_{\mathcal{L}(K_1;K_2)}(\|\cdot\|_{\mathcal{L}_2(K_1;K_2)})$. For $p\in \mathbb{Z}^+$, set
\begin{eqnarray*}
 |\phi|_p=\left\{
            \begin{array}{ll}
              \Big(\int_{\mathcal{M}}|\phi(x,z)|^pdxdz\Big)^{\frac{1}{p}}, &\forall \phi\in L^p(\mathcal{M}), \\
               \Big(\int^l_{0}|\phi(x)|^pdx\Big)^{\frac{1}{p}}, & \forall \phi\in L^p((0,L)).
            \end{array}
          \right.
\end{eqnarray*}
In particular, $|\cdot|$ and $(\cdot,\cdot)$ represent norm and inner product of $L^2(\mathcal{M})$ (or $L^2((0,L))$), respectively. For $m\in \mathbb{N}_+$, $(W^{m,p}(\mathcal{M}), \|\cdot\|_{m,p})$ stands for the classical Sobolev space, see \cite{Adams}. When $p=2$, we denote by $H^m(\mathcal{M})=W^{m,2}(\mathcal{M})$,
\begin{equation}\notag
\left\{
  \begin{array}{ll}
    H^{m}(\mathcal{M})=\Big\{v\Big| \partial_{\alpha}v\in (L^2(\mathcal{M}))^2\ {\rm for} \ |\alpha|\leq m\Big\},&  \\
    |v|^2_{H^{m}(\mathcal{M})}=\sum_{0\leq|\alpha|\leq m}|\partial_{\alpha}v|^2. &
  \end{array}
\right.
\end{equation}
It's known that $(H^{m}(\mathcal{M}), |\cdot|_{H^{m}(\mathcal{M})})$ is a Hilbert space. $|\cdot|_{H^p((0,L))}$ stands for the norm of $H^p((0,L))$ for $p\in \mathbb{Z}^+$.

 Define our working spaces for (\ref{eq5-1})-(\ref{eq-11-1})
 \begin{eqnarray*}
 &&H:=\left\{v\in L^{2}(\mathcal{M})^2:\  \int^{0}_{-h} v dz=0  \right\},\\
 &&V:=\left\{v\in (H^{1}(\mathcal{M}))^2:\  \int^{0}_{-h} v dz=0, \ v=0 \ {\rm{on}}\ \Gamma_l \right\},
 \end{eqnarray*}
The space $H$ is endowed with the $L^2$ inner product
\begin{eqnarray*}
(v,\tilde{v})=\int_{\mathcal{M}}v\tilde{v}dxdz.
\end{eqnarray*}
The norm of $H$ is denoted by $|v|=(v,v)^{\frac{1}{2}}$.
The inner product and norm in the space $V$ are given by
\begin{eqnarray*}
((v,\tilde{v}))&=&\int_{\mathcal{M}}(\partial_x v  \partial_x \tilde{v}+\partial_z v  \partial_z \tilde{v})dxdz,
\end{eqnarray*}
and taking $\|\cdot\|=\sqrt{((\cdot,\cdot))}$.
Note that under the above definition, a Poincar$\acute{e}$ inequality $|v|\leq C\|v\|$ holds for all $v\in V$.


Define the intermediate space
\[
\mathcal{H}=\{v\in H, \partial_z v\in H\}.
\]

Let $V'$ be the dual space of $V$. We have the dense and continuous embeddings
\[
V\hookrightarrow H = H'\hookrightarrow V',
\]
and denote by $\langle x, y\rangle$ the duality between $x\in V$ and $y\in V'$.
\subsection{Some functionals}\label{st-1}
 The Leray operator $P_H$ is the orthogonal projection of $L^{2}(\mathcal{M})$ onto $H$. Define a Stokes-type operator $A$ as a bounded map from $V$ to $V'$ as $\langle v, Au \rangle=((v,u))$. $A$ can be extends to an unbounded operator from $H$ to $H$ according to $Av=-P_H \Delta v$ with the domain:
 \[
 D(A)=\Big\{v\in H^2(\mathcal{M}): \int^{0}_{-h} v dz=0,\ v=0 \ {\rm{on}}\ \Gamma_l,\ \partial_z v=0\  {\rm{on}}\ \Gamma_u\cup \Gamma_b\Big\}.
 \]
It's well-known that $A$ is a self-adjoint and positive definite operator. Due to the regularity results of the Stokes problem of geophysical fluid dynamics, we have $|Av|\cong |v|_{H^2(\mathcal{O})}$, see \cite{Z}.

For the nonlinear terms, let
\[
B(v, \tilde{v}):=P_H(v\partial_x\tilde{v}+\Phi(v)\partial_z\tilde{v} ).
\]
We establish that $B$ is a well-defined and continuous mapping from $V\times V\rightarrow V'$ according to
\[
\langle B(u,v),\phi\rangle=b(u,v,\phi),
\]
where the associated trilinear form is given by
\[
b(u,v,\phi):=\int_{\mathcal{M}}(u\partial_x v \phi+\Phi(u)\partial_z v\phi)d\mathcal{M}.
\]
This is contained in the following lemma, which is established in \cite{Ziane}.
\begin{lemma}\label{lemma-1}
 $b$ is a continuous linear form on $V\times V\times V$ and satisfies
  \begin{eqnarray}\label{ee-3}
|b(v,\tilde{v},\hat{v})|=|\langle B(v,\tilde{v}), \hat{v}\rangle|&\leq& C\|\tilde{v}\||v|^{\frac{1}{2}}\|v\|^{\frac{1}{2}}|\hat{v}|^{\frac{1}{2}}\|\hat{v}\|^{\frac{1}{2}}+C|\partial_z \tilde{v}|\|v\||\hat{v}|^{\frac{1}{2}}\|\hat{v}\|^{\frac{1}{2}}\\
\label{ee-4}
|\langle B(\tilde{v},\tilde{v})-B(\hat{v},\hat{v}), \tilde{v}-\hat{v}\rangle|&\leq& C\|\tilde{v}\||\tilde{v}-\hat{v}|\|\tilde{v}-\hat{v}\|+C|\partial_z \tilde{v}|\|\tilde{v}-\hat{v}\|^{\frac{3}{2}}|\tilde{v}-\hat{v}|^{\frac{1}{2}},
\end{eqnarray}
for any $v, \tilde{v},\hat{v} \in V$. Moreover, $b$ satisfies the cancellation property $b(u,v,v)=0$ and
 \begin{eqnarray}\label{ee-2}
b(v,\tilde{v},\hat{v})=-b(v,\hat{v},\tilde{v}).
  \end{eqnarray}
\end{lemma}

\begin{remark}
The above estimates of nonlinear terms are of higher order than 2D Navier-Stokes equations in \cite{B-B}, which results in difficulties stated in the introduction.
\end{remark}

For the stochastic forcing, we fix a single stochastic basis $\mathcal{T}:=(\Omega, \mathcal{F}, \{\mathcal{F}_t\}_{t\geq 0}, \mathbb{P}, W)$ with the expectation $\mathbb{E}$. Here, $W$
 is a cylindrical Wiener process with the form $W(t,\omega)=\sum_{i\geq1}r_iw_i(t,\omega)$, where $\{r_i\}_{i\geq 1}$ is a complete orthonormal basis of a Hilbert space
$U$, $\{w_i\}_{i\geq1}$ is a sequence of independent one-dimensional standard Brownian motions on $(\Omega, \mathcal{F}, \{\mathcal{F}_t\}_{t\geq 0}, \mathbb{P})$.

Set
\[
F(t,v(t))=Av(t)+B(v(t),v(t)),
\]
using the above functionals, we obtain
\begin{eqnarray}\label{equ-7}
\left\{
  \begin{array}{ll}
    dv(t)+F(t,v(t))dt=\psi(t,v(t)) dW(t), \\
    v(0)=v_0.
  \end{array}
\right.
\end{eqnarray}

\section{Global well-posedness}\label{s-1}
In this part, we aim to obtain a priori estimates of the strong solution of (\ref{equ-7}). Firstly, we introduce the following definition stated in \cite{Ziane}.
\begin{dfn}\label{dfn-3}
Let $\mathcal{T}=(\Omega, \mathcal{F}, \{\mathcal{F}_t\}_{t\geq 0}, \mathbb{P}, W)$ be a fixed stochastic basis, $T>0$ and $p\geq 2$. Assume the initial data $v_0\in L^p(\Omega;H)$ and is $\mathcal{F}_0-$measurable.
An $\mathcal{F}_t-$predictable stochasitc process  $v(t, \omega)$ is called a strong solution of (\ref{equ-7}) on $[0,T]$ with the initial value $v_0$ if
\[
v\in C([0,T];H)\quad P-a.s. \quad v\in L^p(\Omega; C([0,T];H))\bigcap L^p(\Omega; L^2([0,T];V)),
\]
and satisfies
\[
 (v(t),\phi)-(v_0,\phi)+\int^{t}_{0}\Big[\langle v(s), A \phi\rangle+\langle B(v,v), \phi\rangle \Big]ds=\int^{t}_{0}(\psi(s,v(s)) dW(s),\phi), \quad P-a.s.
 \]
for all $\phi\in D(A)$.
\end{dfn}
In order to obtain the global well-posedness of (\ref{equ-7}), we need the following Hypotheses:
\begin{flushleft}
\begin{description}
  \item[\textbf{Hypothesis A:}] $\psi$ is a continuous mapping, $\psi: [0,T]\times V\rightarrow \mathcal{L}_2({U}; H)$ (resp. $\psi: [0,T]\times H\rightarrow \mathcal{L}_2({U}; H)$ for $\varepsilon=0$) satisfies that there exist positive constants $K_i, i=0, \cdot\cdot\cdot, 4$, such that for $t\in[0,T]$, $0\leq\varepsilon<1$,
\begin{description}
  \item[(A.1)] $\|\psi(t, \phi)\|^2_{\mathcal{L}_2({U}; H)}\leq K_0+K_1|\phi|^2+\varepsilon K_2\|\phi\|^2, \quad \phi\in V$;
  \item[(A.2)] $\|\psi(t, \phi_1)-\psi(t,\phi_2)\|^2_{\mathcal{L}_2({U}; H)}\leq K_3|\phi_1-\phi_2|^2+\varepsilon K_4\|\phi_1-\phi_2\|^2, \quad  \phi_1, \phi_2 \in V$.
\end{description}
\end{description}
\textbf{Hypothesis B:} There exist constants $L_i, i=0,\cdot\cdot\cdot,2$, such that for $t\in[0,T]$, $0\leq\varepsilon<1$,
\begin{eqnarray*}
 \|\partial_z \psi(t, \phi)\|^2_{\mathcal{L}_2({U}; H)}&\leq& L_0+L_1|\partial_z\phi|^2+\varepsilon L_2\|\partial_z\phi\|^2, \quad \partial_z\phi\in V.
\end{eqnarray*}
\end{flushleft}


\begin{thm}\label{thm-3}
Assume $v_0\in \mathcal{H}$, \textbf{Hypotheses A, B} hold with $K_2\leq K_4<2$ and $L_2<2$, there exists a unique global solution $v$ of (\ref{equ-7}) in the sense of Definition \ref{dfn-3} with $v(0)=v_0$. Furthermore, if $q\in [2, \frac{1}{37}+\frac{2}{37K_2})$, there exists a constant $C=C(\varepsilon, q,K_0,K_1,K_2,T)$ such that
\begin{eqnarray}\label{ee-5}
\mathbb{E}\Big(\sup_{0\leq s\leq T}|v(s)|^q+\int^T_0\|v(s)\|^2|v(s)|^{q-2}ds\Big)\leq C(1+\mathbb{E}|v_0|^q).
\end{eqnarray}
If $K_2< \frac{2}{147}$, then
\begin{eqnarray}\label{qq-0}
\mathbb{E}\int^T_0|v(s)|^2\|v(s)\|^2ds\leq C(1+\mathbb{E}|v_0|^4).
\end{eqnarray}
Similarly, if $q\in [2, \frac{1}{37}+\frac{2}{37L_2})$, there exists a constant $C=C(\varepsilon, q,R_0,R_1,R_2,T)$ such that
\begin{eqnarray*}
\mathbb{E}\Big(\sup_{0\leq s\leq T}|\partial_zv(s)|^q+\int^T_0\|\partial_zv(s)\|^2|\partial_zv(s)|^{q-2}ds\Big)\leq C(1+\mathbb{E}|\partial_zv_0|^q).
\end{eqnarray*}
In particular, if $L_2< \frac{2}{147}$, we have
\begin{eqnarray*}
\mathbb{E}\Big(\sup_{0\leq s\leq T}|\partial_zv(s)|^4+\int^T_0|\partial_zv(s)|^2\|\partial_zv(s)\|^2ds\Big)\leq C(1+\mathbb{E}|\partial_zv_0|^4).
\end{eqnarray*}
\end{thm}
\begin{proof}
 When $K_2\leq K_4<2$ and $L_2<2$, the global well-posedness of strong solution to (\ref{equ-7}) in the sense of Definition \ref{dfn-3} has been proved by \cite{Ziane}, we omit it. Let $v$ be the strong solution of (\ref{equ-7}).
For any $q\geq 2$, applying It\^{o} formula to $|v(t)|^q$, we have
\begin{eqnarray*}
&&d|v(t)|^q+q|v(t)|^{q-2}\|v(t)\|^2dt\\
&=& -q|v(t)|^{q-2}\langle v(t), B(v(t),v(t))\rangle dt\\
&&\
+q|v(t)|^{q-2}\langle v(t), \psi(t,v(t)) dW(t)\rangle  +\frac{q(q-1)}{2}|v(t)|^{q-2}\|\psi(t,v(t))\|^2_{\mathcal{L}_2(U;H)}dt.
\end{eqnarray*}
Using (\ref{ee-2}), we obtain
\begin{eqnarray}\notag
&&d|v(t)|^q+q|v(t)|^{q-2}\|v(t)\|^2dt\\
\label{eee-9}
&\leq&
q|v(t)|^{q-2}\langle v(t), \psi(t,v(t)) dW(t)\rangle +\frac{q(q-1)}{2}|v(t)|^{q-2}\|\psi(t,v(t))\|^2_{\mathcal{L}_2(U;H)}dt.
\end{eqnarray}
Then,
\begin{eqnarray*}
&&\mathbb{E}\sup_{t\in[0,T]}|v(t)|^q+q\mathbb{E}\int^T_0|v(t)|^{q-2}\|v(t)\|^2dt\\
&\leq&q\mathbb{E}\sup_{t\in[0,T]}\int^t_0|v(s)|^{q-2}\langle v(s), \psi(s,v(s)) dW(s)\rangle \\
&&\ +\frac{q(q-1)}{2}\mathbb{E}\int^T_0|v(t)|^{q-2}\|\psi(t,v(t))\|^2_{\mathcal{L}_2(U;H)}dt\\
&:=& I_1+I_2.
\end{eqnarray*}
With the help of \textbf{Hypothesis A}, we get
\begin{eqnarray*}
I_2
&\leq& \frac{q(q-1)}{2}K_0\mathbb{E}\int^T_0|v(t)|^{q-2}dt+\frac{q(q-1)}{2}K_1\mathbb{E}\int^T_0\sup_{s\in [0,t]}|v(s)|^{q}dt\\
&&\
+\varepsilon\frac{q(q-1)}{2} K_2\mathbb{E}\int^T_0|v(t)|^{q-2}\|v(t)\|^2dt.
\end{eqnarray*}
Utilizing the Burkholder-Davies-Gundy inequality and \textbf{Hypothesis A}, we have
 \begin{eqnarray*}
I_1
&\leq& 6q \mathbb{E}\left(\int^T_0|v(t)|^{2(q-2)}|v(t)|^{2}\|\psi(t,v(t))\|^2_{\mathcal{L}_2(U;H)}dt\right)^{\frac{1}{2}}\\
&\leq& 6q \mathbb{E}\left(\int^T_0|v(t)|^{2q-2}(K_0+K_1|v(t)|^2+\varepsilon K_2\|v(t)\|^2)dt\right)^{\frac{1}{2}}\\
&\leq& 6q \mathbb{E}\left(K_0 \int^T_0|v(t)|^{2q-2}dt\right)^{\frac{1}{2}}+6q \mathbb{E}\left(K_1\int^T_0|v(t)|^{2q}dt\right)^{\frac{1}{2}}\\
&&\ +6q\mathbb{E}\left( \varepsilon K_2\int^T_0|v(t)|^{2q-2}\|v(t)\|^2dt\right)^{\frac{1}{2}}\\
&:=& I^1_1+I^2_1+ I^3_1.
\end{eqnarray*}
Using the Young inequality, it gives
\begin{eqnarray*}
I^1_1&\leq& 6q K^{\frac{1}{2}}_0\mathbb{E}\left( \sup_{t\in [0,T]}|v(t)|^{q}\int^T_0|v(t)|^{q-2}dt\right)^{\frac{1}{2}}\\
&\leq& 6q K^{\frac{1}{2}}_0\mathbb{E}\left[ \sup_{t\in [0,T]}|v(t)|^{\frac{q}{2}}\left(\int^T_0|v(t)|^{q-2}dt\right)^{\frac{1}{2}}\right]\\
&\leq& \frac{1}{6}\mathbb{E}( \sup_{t\in [0,T]}|v(t)|^{q})+ 18q^2K_0 \mathbb{E}\int^T_0\sup_{s\in[0,t]}|v(s)|^{q}dt+18q^2K_0T.
\end{eqnarray*}
$I^2_1$ can be bounded as
\begin{eqnarray*}
I^2_1&\leq& 6q K^{\frac{1}{2}}_1\mathbb{E}\left(\int^T_0|v(t)|^{2q}dt\right)^{\frac{1}{2}}\\
&\leq& \frac{1}{6}\mathbb{E}( \sup_{t\in [0,T]}|v(t)|^{q})+ 18q^2K_1 \mathbb{E}\int^T_0\sup_{s\in[0,t]}|v(s)|^{q}dt.
\end{eqnarray*}
By the Young inequality, we have
\begin{eqnarray*}
I^3_1&\leq& 6q \varepsilon^{\frac{1}{2}} K^{\frac{1}{2}}_2\mathbb{E}\left( \sup_{t\in [0,T]}|v(t)|^q\int^T_0|v(t)|^{q-2}\|v(t)\|^2dt\right)^{\frac{1}{2}}\\
&\leq& \frac{1}{6}\mathbb{E}( \sup_{t\in [0,T]}|v(t)|^{q})+ 18q^2\varepsilon K_2 \mathbb{E}\int^T_0|v(t)|^{q-2}\|v(t)\|^2dt.
\end{eqnarray*}
Based on the above inequalities, we have
\begin{eqnarray*}
I_1
&\leq& \frac{1}{2}\mathbb{E}( \sup_{t\in [0,T]}|v(t)|^{q})+ 18q^2(K_0 +K_1) \mathbb{E}\int^T_0\sup_{s\in[0,t]}|v(s)|^{q}dt\\
&&\ +
18q^2\varepsilon K_2 \mathbb{E}\int^T_0|v(t)|^{q-2}\|v(t)\|^2dt+18q^2K_0T.
\end{eqnarray*}
Collecting the above estimates, we conclude that
\begin{eqnarray}\notag
&&\mathbb{E}\sup_{t\in[0,T]}|v(t)|^q+2\left(q-\varepsilon\frac{q(q-1)}{2}K_2-18q^2\varepsilon K_2\right)\mathbb{E}\int^T_0|v(t)|^{q-2}\|v(t)\|^2dt\\ \notag
&\leq&2\left(\frac{q(q-1)}{2}K_0+\frac{q(q-1)}{2}K_1+ 18q^2(K_0 +K_1)\right) \mathbb{E}\int^T_0\sup_{s\in[0,t]}|v(s)|^{q}dt\\
\label{r-1}
&&\ +q(q-1)K_0 T
+36q^2K_0T.
\end{eqnarray}
When $q\in [2, \frac{1}{37}+\frac{2}{37K_2})$, we have
\[
q-\varepsilon\frac{q(q-1)}{2}K_2-18q^2\varepsilon K_2>0.
\]
Applying Gronwall inequality to (\ref{r-1}), we obtain
\begin{eqnarray}\label{r-2}
\mathbb{E}\sup_{t\in[0,T]}|v(t)|^q\leq C(q,K_0,K_1,K_2,T)(1+\mathbb{E}|v_0|^q).
\end{eqnarray}
Combining (\ref{r-1}) and (\ref{r-2}), we get
\begin{eqnarray}\label{r-3}
\mathbb{E}\sup_{t\in[0,T]}|v(t)|^q+\mathbb{E}\int^T_0|v(t)|^{q-2}\|v(t)\|^2dt\leq C(q,K_0,K_1,K_2,T)(1+\mathbb{E}|v_0|^q).
\end{eqnarray}
Let $r=\partial_z v$. From (\ref{equ-7}), we have
\begin{eqnarray}\label{eee-7}
dr+Ar dt+(v\partial_x r+\Phi(v)\partial_z r)dt=\partial_z \psi(t, v(t))dW(t).
\end{eqnarray}
Applying It\^{o} formula to (\ref{eee-7}), we obtain
\begin{eqnarray*}
&&d|r(t)|^q+q|r(t)|^{q-2}\|r(t)\|^2dt\\
&=& -q|r(t)|^{q-2}\langle r(t),(v\partial_x r+\Phi(v)\partial_z r) \rangle dt\\
&&\
+q|r(t)|^{q-2}\langle r(t), \partial_z\psi(t,v(t)) dW(t)\rangle  +\frac{q(q-1)}{2}|r(t)|^{q-2}\|\partial_z\psi(t,v(t))\|^2_{\mathcal{L}_2(U;H)}dt.
\end{eqnarray*}
We deduce from (\ref{ee-2}) that
\begin{eqnarray}\notag
&&d|r(t)|^q+q|r(t)|^{q-2}\|r(t)\|^2dt\\
\label{eee-8}
&=& q|r(t)|^{q-2}\langle r(t), \partial_z\psi(t,v(t)) dW(t)\rangle  +\frac{q(q-1)}{2}|r(t)|^{q-2}\|\partial_z\psi(t,v(t))\|^2_{\mathcal{L}_2(U;H)}dt.
\end{eqnarray}
Note that (\ref{eee-8}) is similar to (\ref{eee-9}). Following the same process exactly as above, we conclude the result.

\end{proof}
\begin{remark}
For (\ref{equ-7}) with $v_0\in L^p(\Omega;V)$, we have no uniform $V-$norm estimates $\mathbb{E}\sup_{t\in[0,T]}\|v(t)\|^2\leq C$.
\end{remark}

\section{Splitting up method}
Let $\prod^n=\{0=t_0<t_1<\cdot\cdot\cdot<t_n=T\}$ be a finite partition of a given interval $[0,T]$ with a constant mesh $h=\frac{T}{n}$. Let $\varepsilon \in [0,1)$ and let $F_\varepsilon:[0,T]\times V\rightarrow V'$ be defined by
\[
F_\varepsilon(t,v)=(1-\varepsilon)Av +B(v,v).
\]
It's easy to know $F_0(t,v)=F(t,v)$.

Set $t_{-1}=-\frac{T}{n}$. For $t\in [t_{-1}, 0)$, define
\[
v^n(t)=\eta^n(t)=v_0,\quad  \mathcal{F}_t=\mathcal{F}_0.
\]
The scheme is defined by induction as follows. Suppose we have defined processes $v^n(t)$ and $\eta^n(t)$ for $t\in [t_{i-1}, t_i)$, $i=0,\cdot\cdot\cdot, n-1$, such that $\eta^n(t^{-}_i)$ is an $H-$valued $\mathcal{F}_{t_i}-$measurable function. This clearly holds for $i=0$. Then we define $v^n(t), t\in[t_i,t_{i+1})$ as the unique solution of the (deterministic) problem with positive viscosity $1-\varepsilon$ and with initial condition $\eta^n(t^{-}_i)$ at time $t_i$, that is,
\begin{eqnarray}\label{qq-1}
\left\{
  \begin{array}{ll}
    \frac{d v^n(t)}{dt}+F_\varepsilon(t, v^n(t))=0, & t\in [t_i, t_{i+1}), \\
    v^n(t_i)=v^n(t^+_i)=\eta^n(t^-_i), &
  \end{array}
\right.
\end{eqnarray}
Note that $v^n(t^-_{i+1})$ is a well-defined $H-$valued $\mathcal{F}_{t_i}-$measurable random variable.
Then we can define $\eta^n(t), t\in [t_i, t_{i+1})$ as the unique solution of the random problem with initial condition $v^n(t^-_{i+1})$ at time $t_i$:
\begin{eqnarray}\label{qq-2}
\left\{
  \begin{array}{ll}
   d \eta^n(t)+\varepsilon A \eta^n(t)dt=\psi(t, \eta^n(t))dW(t), & t\in [t_i, t_{i+1}), \\
    \eta^n(t_i)=\eta^n(t^+_i)=v^n(t^-_{i+1}), &
  \end{array}
\right.
\end{eqnarray}
We claim that $\eta^n(t^-_{i+1})$ defined above is a well-defined  $H-$valued $\mathcal{F}_{t_{i+1}}-$measurable random variable. In fact, when $\varepsilon>0$,  it's classical that (\ref{qq-2}) has a unique weak solution provided the stochastic parabolic condition holds ($ K_2$, $K_4$, $L_2$  are small enough). When $\varepsilon=0$, the smoothing effect of $A$ does not act anymore, but $\psi$ satisfies the usual growth and Lipschitz conditions for the $H-$norm. Finally, let $v^n(T^+)=\eta^n(T^-)$.

\begin{remark}
As stated in \cite{B-B}, $v^n$ and $\eta^n$ constructed above are not continuous, only right continuous.
\end{remark}
In order to prove the convergence of the above scheme, we will need to establish a priori estimates on $v^n$ and $\eta^n$. Firstly, we introduce some notations. Recall that $\prod^n=\{0=t_0<t_1<\cdot\cdot\cdot<t_n=T\}$. Set
\begin{eqnarray}\label{qq-3}
\left\{
  \begin{array}{ll}
  d_n(t):=t_i,\ d^*_n(t):= t_{i+1} , & {\rm{for}}\ t\in [t_i, t_{i+1}), \quad i=0,1, \cdot\cdot\cdot, n-2, \\
   d_n(t):=t_{n-1},\ d^*_n(t):= t_n , & {\rm{for}}\ t\in [t_{n-1}, t_n].
  \end{array}
\right.
\end{eqnarray}
Then, the processes $v^n(t), \eta^n(t)$ can be rewritten in a way close to the continuous equation:
\begin{eqnarray}\label{qq-4}
v^n(t)=v_0-\int^t_0 F_\varepsilon(s, v^n(s))ds+\int^{d_n(t)}_0[-\varepsilon A\eta^n(s)ds+\psi(s,\eta^n(s))dW(s)],\\
\label{qq-5}
\eta^n(t)=v_0-\int^{d^*_n(t)}_0 F_\varepsilon(s, v^n(s))ds+\int^{t}_0[-\varepsilon A\eta^n(s)ds+\psi(s,\eta^n(s))dW(s)].
\end{eqnarray}

\par
In the following, we aim to establish both $H-$norm and $V-$norm of the difference between $v^n$ and $\eta^n$.
\subsection{$H-$norm of the difference between $v^n$ and $\eta^n$}
Firstly, we need to obtain a priori estimates on $v^n$ and $\eta^n$.
\begin{lemma}\label{lem-3}
Let $v_0\in \mathcal{H}$. Fix $\varepsilon\in [0,1)$. Let $\textbf{Hypotheses A, B}$ hold with $K_2\leq K_4< 2$ and $L_2<2$. Then there exists a positive constant $C=C(\varepsilon,T,\mathbb{E}|v_0|^2,K_i, L_i)$ such that for every integer $n\geq 1$,
\begin{eqnarray}\label{qq-6}
\sup_{t\in[0,T]}\mathbb{E}\Big(|\eta^n(t)|^2+\sup_{s\in[d_n(t),d^*_n(t))}|v^n(s)|^2\Big)+\mathbb{E}\int^T_0\|v^n(s)\|^2ds\leq C.
\end{eqnarray}
Moreover, if $\varepsilon\in (0,1)$, there exists a constant C such that
\begin{eqnarray}\label{qq-7}
\sup_{n}\mathbb{E}\int^T_0\|\eta^n(s)\|^2ds\leq C.
\end{eqnarray}
\end{lemma}
\begin{proof}
Taking the scalar product of (\ref{qq-1}) by $v^n$ and integrating over $(t_i, t]$ for $t\in[t_i, t_{i+1})$, we have
\begin{eqnarray*}
|v^n(t)|^2+2(1-\varepsilon)\int^t_{t_i}\|v^n(s)\|^2ds=|\eta^n(t^-_i)|^2
-2\int^t_{t_i}\langle B(v^n(s), v^n(s)), v^n(s)\rangle ds.
\end{eqnarray*}
By (\ref{ee-2}), we obtain
\begin{eqnarray}\label{qq-10}
|v^n(t)|^2+2(1-\varepsilon)\int^t_{t_i}\|v^n(s)\|^2ds\leq|\eta^n(t^-_i)|^2.
\end{eqnarray}
Taking expectation of (\ref{qq-10}), we get
\begin{eqnarray}\label{qq-11}
\mathbb{E}(\sup_{t_i\leq t<t_{i+1}}|v^n(t)|^2)\leq \mathbb{E}|\eta^n(t^-_i)|^2.
\end{eqnarray}
Applying It\^{o} formula to (\ref{qq-2}) and by $\textbf{Hypothesis A}$, it yields that for $t\in [t_i, t_{i+1})$,
\begin{eqnarray*}
\mathbb{E}|\eta^n(t)|^2+2\varepsilon \mathbb{E}\int^t_{t_i}\|\eta^n(s)\|^2ds&=& \mathbb{E}|v^n(t^-_{i+1})|^2+\mathbb{E}\int^t_{t_i}\|\psi(s, \eta^n(s))\|^2_{\mathcal{L}_2(U;H)}ds\\
&\leq& \mathbb{E}|v^n(t^-_{i+1})|^2+\mathbb{E}\int^t_{t_i}(K_0+K_1|\eta^n(s)|^2+\varepsilon K_2\|\eta^n(s)\|^2)ds.
\end{eqnarray*}
Then
\begin{eqnarray}\label{qq-12}
\mathbb{E}|\eta^n(t)|^2+\varepsilon(2-K_2)\mathbb{E} \int^t_{t_i}\|\eta^n(s)\|^2ds
\leq \mathbb{E}|v^n(t^-_{i+1})|^2+ \frac{K_0 T}{n}+K_1\int^t_{t_i}\mathbb{E}|\eta^n(s)|^2ds.
\end{eqnarray}
Since $K_2<2$, we can neglect the integral of $V-$norm in (\ref{qq-12}) to obtain
\begin{eqnarray}\label{qq-13}
\sup_{t_i\leq t<t_{i+1}}\mathbb{E}|\eta^n(t)|^2
\leq (\mathbb{E}|v^n(t^-_{i+1})|^2+ \frac{K_0 T}{n}){\rm{e}}^{\frac{K_1 T}{n}}.
\end{eqnarray}
Putting (\ref{qq-11}) to (\ref{qq-13}), it gives
\begin{eqnarray}\label{qq-14}
\sup_{t_i\leq t<t_{i+1}}\mathbb{E}|\eta^n(t)|^2
\leq (\mathbb{E}|\eta^n(t^-_i)|^2+ \frac{K_0 T}{n}){\rm{e}}^{\frac{K_1 T}{n}}.
\end{eqnarray}
Set
\[
\tilde{r}_1:=K_1,\quad \tilde{r}_2:=K_0,
\]
then, by a mathematical induction argument, we infer that for $i=0, \cdot\cdot\cdot, n-1$,
\begin{eqnarray*}
\mathbb{E}(\sup_{t_i\leq t<t_{i+1}}|v^n(t)|^2)\vee (\sup_{t_i\leq t<t_{i+1}}\mathbb{E}|\eta^n(t)|^2)
\leq \mathbb{E}|v_0|^2 {\rm{e}}^{(i+1)\frac{\tilde{r}_1 T}{n}}+\frac{\tilde{r}_2 T}{n}\sum^{i+1}_{j=1}{\rm{e}}^{j \frac{\tilde{r}_1T}{n}}.
\end{eqnarray*}
Hence, we deduce that
\begin{eqnarray}\notag
\Big[\sup_{t\in[0,T]}\mathbb{E}(\sup_{d_n(t)\leq s<d^*_n(t)}|v^n(s)|^2)\Big]\vee \Big[\sup_{t\in[0,T]}\mathbb{E}|\eta^n(t)|^2\Big]
&\leq& \mathbb{E}|v_0|^2 {\rm{e}}^{\tilde{r}_1T}+\frac{\tilde{r}_2 T}{n}\sum^{n}_{j=1}{\rm{e}}^{j \frac{\tilde{r}_1T}{n}}\\
\label{qq-15}
&\leq& \mathbb{E}|v_0|^2 {\rm{e}}^{\tilde{r}_1T}+\frac{\tilde{r}_2}{\tilde{r}_1} {\rm{e}}^{2\tilde{r}_1T},
\end{eqnarray}
which proves part of (\ref{qq-6}). Moreover, from (\ref{qq-10}), (\ref{qq-12}), and using (\ref{qq-15}), we obtain that for every $i=0, \cdot\cdot\cdot, n-1$,
\begin{eqnarray*}
\mathbb{E}|v^n(t^-_{i+1})|^2+(1-\varepsilon)\mathbb{E}\int^{t_{i+1}}_{t_i}\|v^n(s)\|^2ds\leq \mathbb{E}|\eta^n(t^-_i)|^2
+\frac{CT}{n},\\
\mathbb{E}|\eta^n(t^-_{i+1})|^2+\varepsilon(2-K_2) \int^{t_{i+1}}_{t_i}\|\eta^n(s)\|^2ds
\leq \mathbb{E}|v^n(t^-_{i+1})|^2+ \frac{C T}{n}.
\end{eqnarray*}
Adding all these inequalities from $i=0$ to $n-1$, we conclude the proof of (\ref{qq-6}). At the same time, when $\varepsilon >0$, it gives (\ref{qq-7}).
\end{proof}

Referring to \cite{B-B} and similar to Lemma \ref{lem-3}, we have the following higher moments of $H-$norm.
\begin{lemma}\label{lem-4}
Let $v_0\in \mathcal{H}$ be $\mathcal{F}_0-$measurable. Fix $\varepsilon\in [0,1)$. Let $\textbf{Hypotheses A, B}$ hold with $K_2< \frac{2}{2p-1}$, for some $p\geq 2$ and $K_4\leq L_2<2$. Then there exists a positive constant $C=C(\varepsilon,T,\mathbb{E}|v_0|^{2p}, K_i, L_i)$ such that for every integer $n\geq 1$,
\begin{eqnarray}\label{qq-16}
\sup_{t\in[0,T]}\mathbb{E}\Big(|\eta^n(t)|^{2p}+\sup_{s\in[d_n(t),d^*_n(t))}|v^n(s)|^{2p}\Big)+\mathbb{E}\int^T_0\|v^n(s)\|^2|v^n(s)|^{2(p-1)}ds\leq C.
\end{eqnarray}
In particular, when $p=2$, it gives
\begin{eqnarray}\label{eee-10}
\mathbb{E}\int^T_0|v^n(s)|^4ds\leq\mathbb{E}\int^T_0\|v^n(s)\|^2|v^n(s)|^{2}ds\leq C.
\end{eqnarray}
Moreover, if $\varepsilon\in (0,1)$, there exists a constant C such that
\begin{eqnarray}\label{qq-17}
\sup_{n\in \mathbb{N}}\mathbb{E}\int^T_0\|\eta^n(s)\|^2|\eta^n(s)|^{2(p-1)}ds\leq C.
\end{eqnarray}
\end{lemma}
Compared with the 2D stochastic Navier-Stokes equations, we need the additional estimates of $\partial_z v^n$ and $\partial_z \eta^n$.

Define
\[
 r^n=\partial_z v^n,\quad q^n=\partial_z \eta^n.
\]
From (\ref{qq-1}), we have for $t\in [t_i, t_{i+1})$,
\begin{eqnarray}\label{qq-18}
dr^n+(1-\varepsilon)Ar^ndt+\Big(v^n\partial_x r^n+\Phi(v^n)\partial_z r^n\Big)dt=0.
\end{eqnarray}
Moreover, we deduce from (\ref{qq-2}) that for $t\in [t_i, t_{i+1})$,
\begin{eqnarray}\label{qq-20}
    dq^n+\varepsilon Aq^ndt=\partial_z \psi(t, \eta^n(t))dW(t).
\end{eqnarray}
The initial conditions for (\ref{qq-18}) and (\ref{qq-20}) are $r^n(t_i)=q^n(t^-_i)$, $q^n(t_i)=r^n(t^-_{i+1})$, respectively.
\begin{lemma}\label{lem-5}
Let $v_0\in \mathcal{H}$ be $\mathcal{F}_0-$measurable random variable. Fix $\varepsilon\in [0,1)$. Let $\textbf{Hypotheses A, B}$ hold with $K_2\leq K_4\leq L_2< 2$. Then there exists a positive constant $C=C(\varepsilon,T,\mathbb{E}|\partial_zv_0|^2,K_i, L_i)$ such that for every integer $n\geq 1$,
\begin{eqnarray}\label{qq-21}
\sup_{t\in[0,T]}\mathbb{E}\Big(|q^n(t)|^2+\sup_{s\in[d_n(t),d^*_n(t))}|r^n(s)|^2\Big)+\mathbb{E}\int^T_0\|r^n(s)\|^2ds\leq C.
\end{eqnarray}
Moreover, if $\varepsilon\in (0,1)$, there exists a constant C such that
\begin{eqnarray}\label{qq-22}
\sup_{n}\mathbb{E}\int^T_0\|q^n(s)\|^2ds\leq C.
\end{eqnarray}
\end{lemma}

\begin{proof}
Taking the scalar product of (\ref{qq-18}) with $r^n$ in $H$ and integrating over $(t_i, t]$ for $t\in [t_i, t_{i+1})$. With the help of the cancellation property, we have
\begin{eqnarray}\label{qq-23}
\frac{d|r^n|^2}{dt}+2(1-\varepsilon)\|r^n\|^2\leq 0,
\end{eqnarray}
that is,
\begin{eqnarray}\label{qq-25}
|r^n(t)|^2+2(1-\varepsilon)\int^t_{t_i}\|r^n(s)\|^2ds
\leq |q^n(t^-_i)|^2.
\end{eqnarray}
Taking the expectation of (\ref{qq-25}), we deduce that
\begin{eqnarray}\label{qq-26}
\mathbb{E}(\sup_{t\in [t_i, t_{i+1})}|r^n(t)|^2)\leq \mathbb{E}|q^n(t^-_i)|^2.
\end{eqnarray}
Using It\^{o} formula to (\ref{qq-20}) and by \textbf{Hypothesis B}, we have for $t\in [t_i, t_{i+1})$,
\begin{eqnarray*}
\mathbb{E}|q^n(t)|^2+\varepsilon(2-L_2)\mathbb{E}\int^t_{t_i}\|q^n(s)\|^2ds\leq \mathbb{E} |r^n(t^-_{i+1})|^2+\frac{L_0 T}{n}+L_1\int^t_{t_i}\mathbb{E}|q^n(s)|^2ds .
\end{eqnarray*}
When $L_2<2$, ignoring the $V-$norm and by Gronwall inequality, we get
\begin{eqnarray}\label{qq-27}
\sup_{t\in [t_i, t_{i+1})}\mathbb{E}|q^n(t)|^2\leq \Big(\mathbb{E} |r^n(t^-_{i+1})|^2+\frac{L_0 T}{n}\Big){\rm{e}}^{\frac{L_1 T}{n}}.
\end{eqnarray}
Putting (\ref{qq-26}) into (\ref{qq-27}), we obtain
\begin{eqnarray*}
\sup_{t\in [t_i, t_{i+1})}\mathbb{E}|q^n(t)|^2\leq \Big(\mathbb{E}|q^n(t^-_i)|^2+\frac{L_0 T}{n}\Big){\rm{e}}^{\frac{L_1 T}{n}}.
\end{eqnarray*}
Set $\tilde{r}_3=L_1, \tilde{r}_4=L_0$, by the induction argument,
we have for $i=0, \cdot\cdot\cdot, n-1$,
\begin{eqnarray*}
\mathbb{E}(\sup_{t_i\leq t<t_{i+1}}|r^n(t)|^2)\vee (\sup_{t_i\leq t<t_{i+1}}\mathbb{E}|q^n(t)|^2)
\leq \mathbb{E}|\partial_z v_0|^2 {\rm{e}}^{(i+1)\frac{\tilde{r}_3T}{n}}+\frac{\tilde{r}_4 T}{n}\sum^{i+1}_{j=1}{\rm{e}}^{j \frac{\tilde{r}_3T}{n}}.
\end{eqnarray*}
Hence, we deduce that
\begin{eqnarray}\label{qq-28}
\Big[\sup_{t\in[0,T]}\mathbb{E}(\sup_{d_n(t)\leq s<d^*_n(t)}|r^n(s)|^2)\Big]\vee \Big[\sup_{t\in[0,T]}\mathbb{E}|q^n(t)|^2\Big]
\leq \mathbb{E}|\partial_zv_0|^2{\rm{e}}^{\tilde{r}_3T}+\frac{\tilde{r}_4}{\tilde{r}_3} {\rm{e}}^{2\tilde{r}_3T}.
\end{eqnarray}
Following the same process as Lemma \ref{lem-3}, we can conclude the rest result.
\end{proof}

\begin{lemma}\label{lem-6}
Let $v_0\in \mathcal{H}$ be $\mathcal{F}_0-$measurable random variable. Fix $\varepsilon\in [0,1)$. Let $\textbf{Hypotheses A, B}$ hold with $K_2\leq K_4< 2$ and $L_2<\frac{2}{2p-1}$ for some $p\geq 2$. Then there exists a positive constant $C=C(\varepsilon,T,\mathbb{E}|\partial_zv_0|^2, K_i, L_i)$ such that for every integer $n\geq 1$,
\begin{eqnarray}\label{qq-31}
\sup_{t\in[0,T]}\mathbb{E}\Big(|q^n(t)|^{2p}+\sup_{s\in[d_n(t),d^*_n(t))}|r^n(s)|^{2p}\Big)+\mathbb{E}\int^T_0\|r^n(s)\|^2|r^n (s)|^{2(p-1)}ds\leq C.
\end{eqnarray}
Moreover, if $\varepsilon\in (0,1)$, there exists a constant C such that
\begin{eqnarray}\label{qq-32}
\sup_{n\in \mathbb{N}}\mathbb{E}\int^T_0\|q^n(s)\|^2|q^n(s)|^{2(p-1)}ds\leq C.
\end{eqnarray}
\end{lemma}
Based on the above, we are ready to prove an upper bound of the $H-$norm of the difference between $v^n$ and $\eta^n$.
\begin{prp}\label{prp-1}
Let $v_0\in \mathcal{H}$  be $\mathcal{F}_0-$measurable random variable. For any $\varepsilon\in[0,1)$. Assume $\textbf{Hypotheses A, B}$ hold with $K_2< \frac{2}{3},\ K_4< 2$ and $L_2<\frac{2}{3}$, there exists a positive constant $C=C(\varepsilon,T,\mathbb{E}|\partial_zv_0|^4, K_i, L_i)$ such that for any $n\in \mathbb{N}$,
\begin{eqnarray}\label{qq-34}
\mathbb{E}\int^T_0|v^n(t)-\eta^n(t)|^2dt\leq \frac{CT}{n}.
\end{eqnarray}
\end{prp}
\begin{proof}
\textbf{Case 1:} $\varepsilon =0$. For any $t\in [0,T)$, by (\ref{qq-2}) and \textbf{Hypothesis A}, we have
\begin{eqnarray*}
\mathbb{E}|\eta^n(t)-v^n(d^*_n(t))|^2=\mathbb{E}\int^t_{d_n(t)}\|\psi(s,\eta^n(s))\|^2_{\mathcal{L}_2(U;H)}ds\leq \mathbb{E}\int^t_{d_n(t)}(K_0+ K_1|\eta^n(s)|^2)ds.
\end{eqnarray*}
Then, by Fubini's theorem and Lemma \ref{lem-3}, we obtain
\begin{eqnarray}\label{qq-30}
\mathbb{E}\int^T_0|\eta^n(t)-v^n(d^*_n(t))|^2dt\leq C\mathbb{E}\int^T_0(1+|\eta^n(s)|^2)\left(\int^{d^*_n(s)}_sdt\right)ds\leq C\frac{T}{n}.
\end{eqnarray}
From (\ref{qq-1}), we have
\begin{eqnarray*}
|v^n(d^*_n(t)^-)-v^n(t)|^2=2\int^{d^*_n(t)}_t \langle v^n(s)-v^n(t), dv^n(s)\rangle=\sum^2_{i=1}I_i(t),
\end{eqnarray*}
where
\begin{eqnarray*}
I_1(t)&=&-2(1-\varepsilon)\int^{d^*_n(t)}_t\langle v^n(s)-v^n(t), Av^n(s)\rangle ds,\\
I_2(t)&=&-2\int^{d^*_n(t)}_t\langle v^n(s)-v^n(t), B(v^n(s),v^n(s))\rangle ds,
\end{eqnarray*}
Using Lemma \ref{lem-3} and the Young inequality, we have
\begin{eqnarray*}
\Big|\mathbb{E}\int^T_0I_1(t)dt\Big|&=&\Big|(1-\varepsilon)\mathbb{E}\int^T_0\int^{d^*_n(t)}_t(-2\|v^n(s)\|^2+2\|v^n(s)\|\|v^n(t)\|)dsdt\Big|\\
&\leq& \Big|(1-\varepsilon)\mathbb{E}\int^T_0\int^{d^*_n(t)}_t(-2\|v^n(s)\|^2+2\|v^n(s)\|^2+\frac{1}{2}\|v^n(t)\|^2)dsdt\Big|\\
&\leq&  \frac{1-\varepsilon}{2}\mathbb{E}\int^T_0\|v^n(t)\|^2\left(\int^{d^*_n(t)}_t ds\right)dt\leq \frac{CT}{n}.
\end{eqnarray*}
By (\ref{ee-3}), we have
\begin{eqnarray*}
\Big|\mathbb{E}\int^T_0I_2(t)dt\Big|&\leq& 2\mathbb{E}\int^T_0\int^{d^*_n(t)}_t \|v^n(t)\||v^n(s)|^2_{4}dsdt
+2\mathbb{E}\int^T_0\int^{d^*_n(t)}_t|r^n(t)|\|v^n(s)\||v^n(s)|^{\frac{1}{2}}\|v^n(s)\|^{\frac{1}{2}}dsdt\\
&:=&K_1+K_2.
\end{eqnarray*}
By Lemma \ref{lem-4}, we deduce that
\begin{eqnarray*}
K_1&=&2\mathbb{E}\int^T_0\|v^n(t)\|\left(\int^{d^*_n(t)}_t|v^n(s)|^2_{4}ds\right)dt\\
&\leq&2\left(\mathbb{E}\int^T_0\|v^n(t)\|^2dt\right)^{\frac{1}{2}}\left(\mathbb{E}\int^T_0\Big(\int^{d^*_n(t)}_t|v^n(s)|^2_{4}ds\Big)^2dt\right)^{\frac{1}{2}}\\
&\leq&C\left(\mathbb{E}\int^T_0\frac{T}{n}\int^{d^*_n(t)}_t|v^n(s)|^4_{4}dsdt\right)^{\frac{1}{2}}\\
&\leq&C\left(\frac{T}{n}\mathbb{E}\int^T_0|v^n(s)|^4_{4}\Big(\int^{s}_{d_n(s)}dt\Big)ds\right)^{\frac{1}{2}}\\
&\leq&\frac{CT}{n}.
\end{eqnarray*}
By the Cauchy-Schwarz inequality, Fubini's theorem and Lemmas \ref{lem-3}, \ref{lem-4}, \ref{lem-6}, we get
\begin{eqnarray*}
K_2&\leq& \mathbb{E}\int^T_0\int^{d^*_n(t)}_t|r^n(t)|\|v^n(s)\||v^n(s)|^{\frac{1}{2}}\|v^n(s)\|^{\frac{1}{2}}dsdt\\
&\leq& \mathbb{E}\int^T_0\int^{d^*_n(t)}_t|r^n(t)|^2|v^n(s)|\|v^n(s)\|dsdt+\mathbb{E}\int^T_0\int^{d^*_n(t)}_t\|v^n(s)\|^2dsdt\\
&\leq& \mathbb{E}\int^T_0|r^n(t)|^4\int^{d^*_n(t)}_tdsdt+\mathbb{E}\int^T_0\int^{d^*_n(t)}_t(1+|v^n(s)|^2)\|v^n(s)\|^2dsdt\\
&\leq& \mathbb{E}\int^T_0|r^n(t)|^4\int^{d^*_n(t)}_tdsdt+\mathbb{E}\int^T_0(1+|v^n(s)|^2)\|v^n(s)\|^2(\int^s_{d_n(s)}dt)ds\\
&\leq& \frac{CT}{n}.
\end{eqnarray*}
Therefore,
\begin{eqnarray}\label{qq-33}
\mathbb{E}\int^T_0|v^n(d^*_n(t)^-)-v^n(t)|^2dt\leq \frac{CT}{n}.
\end{eqnarray}
Combining (\ref{qq-30}) and (\ref{qq-33}), we conclude the result when $\varepsilon =0$.

\textbf{Case 2:} $\varepsilon\in(0,1)$. For any $t\in [0,T]$, from (\ref{qq-4}) and (\ref{qq-5}), we have
\begin{eqnarray*}
\eta^n(t)-v^n(t)=-\int^{d^*_n(t)}_t F_{\varepsilon}(s,v^n(s))ds-\varepsilon \int^t_{d_n(t)}A\eta^n(s)ds+\int^t_{d_n(t)}\psi(s, \eta^n(s))dW(s).
\end{eqnarray*}
Applying It\^{o} formula to $|\eta^n(t)-v^n(t)|^2$, we obtain
\begin{eqnarray*}
\mathbb{E}\int^T_0|\eta^n(t)-v^n(t)|^2dt=\sum^4_{i=1}J_i,
\end{eqnarray*}
where
\begin{eqnarray*}
J_1(t)&=&-2(1-\varepsilon)\mathbb{E}\int^T_0\int^{d^*_n(t)}_t\langle Av^n(s), \eta^n(s)-v^n(s)\rangle ds dt,\\
J_2(t)&=&-2\mathbb{E}\int^T_0\int^{d^*_n(t)}_t\langle B(v^n(s),v^n(s)), \eta^n(s)-v^n(s)\rangle ds dt,\\
J_3(t)&=&-2\varepsilon\mathbb{E}\int^T_0\int^t_{d_n(t)}\langle A\eta^n(s), \eta^n(s)-v^n(s)\rangle ds dt,\\
J_4(t)&=&\mathbb{E}\int^T_0\int^t_{d_n(t)}\|\psi(s, \eta^n(s))\|^2_{\mathcal{L}_2(U;H)}dsdt.
\end{eqnarray*}
Exactly as Page 12-13 in \cite{B-B}, we have
\begin{eqnarray*}
J_1(t)\leq \frac{C(1-\varepsilon)T}{n},\quad J_3(t)\leq \frac{C\varepsilon T}{n},\quad J_4(t)\leq \frac{CT}{n}.
\end{eqnarray*}
By (\ref{ee-3}), the Cauchy-Schwarz inequality, Fubini's theorem and Lemmas \ref{lem-3},\ref{lem-4} and \ref{lem-6}, we have
\begin{eqnarray*}
J_2(t)&\leq& 2\mathbb{E}\int^T_0\int^{d^*_n(t)}_t(\|\eta^n(s)\||v^n(s)|\|v^n(s)\|+|\partial_z \eta^n(s)|\|v^n(s)\|^{\frac{3}{2}}|v^n(s)|^{\frac{1}{2}})ds\\
&\leq& C\mathbb{E}\int^T_0(\|\eta^n(s)\|^2+|v^n(s)|^2\|v^n(s)\|^2+|q^n(s)|^4+\|v^n(s)\|^{2}+|v^n(s)|^2\|v^n(s)\|^{2})\Big(\int^s_{d_n(s)}dt\Big)ds \\
&\leq&\frac{CT}{n}.
\end{eqnarray*}
The above estimates imply that (\ref{qq-34}) holds when $\varepsilon\in (0,1)$.
\end{proof}

\subsection{$V-$norm of the difference between $v^n$ and $\eta^n$}
Now, we need an additional hypothesis.
\begin{flushleft}
\textbf{Hypothesis C:} There exist constants $R_i, i=0,1,2$, such that for $t\in[0,T]$, $0\leq\varepsilon<1$,
\begin{eqnarray*}
\|\psi(t, \phi)\|^2_{\mathcal{L}_2(U;V)}\leq R_0+R_1\|\phi\|^2+\varepsilon R_2|A\phi|^2, \quad \phi\in D(A).
\end{eqnarray*}
\end{flushleft}

%

Fix $n$, for some $N>0$, define the stopping time
\begin{eqnarray}\label{eq-74}
\tau^N_n:=\inf\left\{t: \sup_{i=0,\cdot\cdot\cdot,n-1} \int^{t_{i+1}\wedge t}_{t_i\wedge t}(|v^n(s)|^2\|v^n(s)\|^2+|r^n(s)|\|r^n(s)\|)ds>\frac{N}{n}\right\}.
\end{eqnarray}
Then, we obtain
\begin{lemma}\label{lem-8}
Let $v_0\in V$. Fix $\varepsilon\in[0,1)$. Let \textbf{Hypotheses A, B, C} hold with $K_2<\frac{2}{3},\ K_4<2$ and $L_2<\frac{2}{3},\ R_2<2$, then there exists a positive constant $C=C(\varepsilon,T,\mathbb{E}\|v_0\|^2, K_i,L_i,R_i)$ such that for any integer $n\geq 1$,
\begin{eqnarray}\label{qq-49}
\sup_{t\in[0,T\wedge\tau^N_n]}\mathbb{E}\Big(\|\eta^n(t)\|^2+\sup_{s\in[d_n(t)\wedge\tau^N_n,d^*_n(t)\wedge\tau^N_n)}\|v^n(s)\|^2\Big)+\mathbb{E}\int^{T\wedge\tau^N_n}_0\|v^n(s)\|^2_2ds\leq C\tilde{K}(N),
\end{eqnarray}
where $\tilde{K}(N)=\frac{1}{N}e^{C(T) N}$.
Moreover, if $\varepsilon\in (0,1)$, we have
\begin{eqnarray}\label{qq-50}
\sup_{n\in \mathbb{N}}\mathbb{E}\int^{T\wedge\tau^N_n}_0\|\eta^n(s)\|^2_2ds\leq C\tilde{K}(N).
\end{eqnarray}
\end{lemma}
\begin{proof}
Taking the scalar product of (\ref{qq-1}) by $Av^n$ in $H$ and integrating over $(t_i, t]$ for $t\in[t_i, t_{i+1})$, we have
\begin{eqnarray*}
\|v^n(t)\|^2+2(1-\varepsilon)\int^t_{t_i}\|v^n(s)\|^2_2ds
=\|\eta^n(t^-_i)\|^2-2\int^t_{t_i}\langle B(v^n(s), v^n(s)), Av^n(s)\rangle ds.
\end{eqnarray*}
Applying the chain rule to $e^{\phi(t)}\|v^n(t)\|^2$, we reach
\begin{eqnarray*}
&&e^{\phi(t)}\|v^n(t)\|^2+2(1-\varepsilon)\int^t_{t_i}e^{\phi(s)}\|v^n(s)\|^2_2ds\\
&=&e^{\phi(t^-_i)}\|\eta^n(t^-_i)\|^2
-2\int^t_{t_i}e^{\phi(s)}\langle B(v^n(s), v^n(s)), Av^n(s)\rangle ds+\int^t_{t_i}\phi'(s)\|v^n(s)\|^2 e^{\phi}ds.
\end{eqnarray*}
Using H\"{o}lder inequality and interpolation inequality, we deduce that
 \begin{eqnarray*}
&& |\langle B(v^n(s), v^n(s)), Av^n(s)\rangle |\\
 &\leq & C|Av^n||v^n|^{\frac{1}{2}}\|v^n\|\|v^n\|^{\frac{1}{2}}_2+C|Av^n|\|v^n\||r^n|^{\frac{1}{2}}\|r^n\|^{\frac{1}{2}}\\
 &\leq & \frac{(1-\varepsilon)}{2}\|v^n\|^2_2+C_1(|v^n|^2\|v^n\|^2+|r^n|\|r^n\|)\|v^n\|^2.
 \end{eqnarray*}
Then, we have
\begin{eqnarray}\notag
&&\mathbb{E}\left(\sup_{t\in[t_i\wedge \tau^N_n, t_{i+1}\wedge \tau^N_n)}e^{\phi(t)}\|v^n(t)\|^2+(1-\varepsilon)\int^{t_{i+1}\wedge \tau^N_n}_{t_i\wedge \tau^N_n}e^{\phi(s)}\|v^n(s)\|^2_2ds\right)\\ \notag
&\leq&\mathbb{E}(e^{\phi(t^-_i\wedge \tau^N_n)}\|\eta^n(t^-_i\wedge \tau^N_n)\|^2)+2C_1\mathbb{E}\int^{ t_{i+1}\wedge \tau^N_n}_{t_i\wedge \tau^N_n}e^{\phi(s)}(|v^n|^2\|v^n\|^2+|r^n|\|r^n\|)\|v^n\|^2ds\\
\label{eee-1}
&&+\int^{ t_{i+1}\wedge \tau^N_n}_{t_i\wedge \tau^N_n}\phi'(s)\|v^n(s)\|^2 e^{\phi}ds.
\end{eqnarray}
Taking the previous estimates into account, for $t\in[t_i\wedge \tau^N_n, t_{i+1}\wedge \tau^N_n]$, set
\[
\phi(t)=-C_1\int^t_{t^-_i\wedge \tau^N_n}(|v^n|^2\|v^n\|^2+|r^n|\|r^n\|)ds,
\]
where $C_1$ is the constant in (\ref{eee-1}).
By Gronwall inequality, we have
\begin{eqnarray}\label{eee-2}
\mathbb{E}\left(\sup_{t\in[t_i\wedge \tau^N_n, t_{i+1}\wedge \tau^N_n)}e^{\phi(t)}\|v^n(t)\|^2\right)
\leq\mathbb{E}(\|\eta^n(t^-_i\wedge \tau^N_n)\|^2).
\end{eqnarray}
Since $e^{\phi(t_{i+1}\wedge \tau^N_n)}\geq e^{-C_1\frac{N}{n}}$ $a.s.$, we deduce from (\ref{eee-2})  that
\begin{eqnarray}\label{eee-3}
\mathbb{E}\left(\sup_{t\in[t_i\wedge \tau^N_n, t_{i+1}\wedge \tau^N_n)}\|v^n(t)\|^2\right)
\leq\mathbb{E}(\|\eta^n(t^-_i\wedge \tau^N_n)\|^2) e^{C_1\frac{N}{n}}.
\end{eqnarray}

Applying It\^{o} formula to (\ref{qq-2}), by \textbf{Hypothesis B}, we have for $t\in [t_i\wedge \tau^N_n, t_{i+1}\wedge \tau^N_n)$,
\begin{eqnarray*}
\mathbb{E}\|\eta^n(t)\|^2+\varepsilon(2-R_2)\mathbb{E}\int^t_{t_i\wedge \tau^N_n}\|\eta^n(s)\|^2_2ds\leq \mathbb{E} \|v^n(t^-_{i+1}\wedge \tau^N_n)\|^2+\frac{R_0 T}{n}+R_1\int^t_{t_i\wedge \tau^N_n}\mathbb{E}|\eta^n(s)|^2ds .
\end{eqnarray*}
When $R_2<2$, we can ignore the $V-$norm. Then, by (\ref{lem-3}) and Gronwall inequality, we get
\begin{eqnarray}\label{qq-53}
\sup_{t\in [t_i\wedge \tau^N_n, t_{i+1}\wedge \tau^N_n)}\mathbb{E}\|\eta^n(t)\|^2\leq \Big(\mathbb{E} \|v^n(t^-_{i+1}\wedge \tau^N_n)\|^2+\frac{R_0 T}{n}\Big){\rm{e}}^{\frac{R_1T}{n}}.
\end{eqnarray}
Putting (\ref{eee-3}) into (\ref{qq-53}), we deduce that
\begin{eqnarray*}
\sup_{t\in [t_i\wedge \tau^N_n, t_{i+1}\wedge \tau^N_n)}\mathbb{E}\|\eta^n(t)\|^2\leq \Big(\mathbb{E}\|\eta^n(t^{-}_i\wedge \tau^N_n)\|^2+\frac{R_0 T}{n}\Big){\rm{e}}^{\frac{(C_1N+R_1T) }{n}}.
\end{eqnarray*}
Set $\tilde{r}_5=C_1N+R_1T, \tilde{r}_6=R_0 T$, by the induction argument,
we have for $i=0, \cdot\cdot\cdot, n-1$,
\begin{eqnarray*}
\mathbb{E}(\sup_{t_i\wedge \tau^N_n\leq t<t_{i+1}\wedge \tau^N_n}\|v^n(t)\|^2)\vee (\sup_{t_i\wedge \tau^N_n\leq t<t_{i+1}\wedge \tau^N_n}\mathbb{E}\|\eta^n(t)\|^2)
\leq \mathbb{E}\| v_0\|^2 {\rm{e}}^{(i+1)\frac{\tilde{r}_5}{n}}+\frac{\tilde{r}_6 }{n}\sum^{i+1}_{j=1}{\rm{e}}^{j \frac{\tilde{r}_5}{n}}.
\end{eqnarray*}
Hence, we deduce that
\begin{eqnarray}\label{qq-55}
\Big[\sup_{t\in[0,T\wedge \tau^N_n]}\mathbb{E}(\sup_{d_n(t)\wedge \tau^N_n\leq s<d^*_n(t)\wedge \tau^N_n}\|v^n(s)\|^2)\Big]\vee \Big[\sup_{t\in[0,T\wedge \tau^N_n]}\mathbb{E}\|\eta^n(t)\|^2\Big]
\leq \mathbb{E}\|v_0\|^2 {\rm{e}}^{\tilde{r}_5}+\frac{\tilde{r}_6}{\tilde{r}_5} {\rm{e}}^{2\tilde{r}_5}.
\end{eqnarray}

Exactly following the same procedure as Lemma \ref{lem-3}, we can obtain the result.
\end{proof}

\begin{lemma}\label{lem-9}
Let $v_0\in V$ be $\mathcal{F}_0-$measurable random variable. Fix $\varepsilon\in [0,1)$. Let $\textbf{Hypotheses A,B,C}$ hold with $K_2<\frac{2}{3},\ K_4<2$ and $L_2<\frac{2}{2p-1},\  R_2<\frac{2}{2p-1}$ for some $p\geq 2$. Then there exists a positive constant $C=C(\varepsilon,T,\mathbb{E}\|v_0\|^2,K_2,K_4,L_2,R_2)$ such that for every integer $n\geq 1$,
\begin{eqnarray}\label{qq-58}
\sup_{t\in[0,T\wedge \tau^N_n]}\mathbb{E}\Big(\|\eta^n(t)\|^{2p}+\sup_{s\in[d_n(t)\wedge \tau^N_n,d^*_n(t)\wedge \tau^N_n)}\|v^n(s)\|^{2p}\Big)
+\mathbb{E}\int^{T\wedge \tau^N_n}_0\|v^n(s)\|^2_2\|v^n(s)\|^{2(p-1)}ds\leq C\tilde{K}(N),
\end{eqnarray}
where $\tilde{K}(N)$ is the same as Lemma \ref{lem-8}.

Moreover, if $\varepsilon\in (0,1)$, we have
\begin{eqnarray}\label{qq-59}
\sup_{n\in \mathbb{N}}\mathbb{E}\int^{T\wedge \tau^N_n}_0\|\eta^n(s)\|^2_2\|\eta^n(s)\|^{2(p-1)}ds\leq C\tilde{K}(N).
\end{eqnarray}
\end{lemma}


Up to now, we are ready to obtain an upper bound of the $V-$norm of the difference between $v^n$ and $\eta^n$.
\begin{prp}\label{prp-2}
Let $v_0\in V$ be $\mathcal{F}_0-$measurable random variable. Fix $\varepsilon\in [0,1)$. Assume \textbf{Hypotheses A, B, C} hold with $K_2<\frac{2}{3},\  K_4< 2$ and $L_2<\frac{2}{3},\  R_2<\frac{2}{3}$, there exists a positive constant $C$ such that for any $n\in \mathbb{N}$,
\begin{eqnarray}\label{qq-56}
\mathbb{E}\int^{T\wedge \tau^N_n}_0\|v^n(t)-\eta^n(t)\|^2dt\leq \frac{C(T)\tilde{K}(N)}{n}.
\end{eqnarray}
\end{prp}
\begin{proof}
\textbf{Case 1:} $\varepsilon =0$. For any $t\in [0,T\wedge\tau^N_n)$, by (\ref{qq-2}) and \textbf{Hypothesis C}, we have
\begin{eqnarray*}
\mathbb{E}\|\eta^n(t)-v^n(d^*_n(t))\|^2=\mathbb{E}\int^t_{d_n(t)}\|\psi(s,\eta^n(s))\|^2_{\mathcal{L}_2(U;V)}ds\leq \mathbb{E}\int^t_{d_n(t)}(R_0+ R_1\|\eta^n(s)\|^2)ds.
\end{eqnarray*}
Then, by Fubini's theorem and Lemma \ref{lem-3}, we obtain
\begin{eqnarray}\label{qq-57}
\mathbb{E}\int^{T\wedge\tau^N_n}_0\|\eta^n(t)-v^n(d^*_n(t))\|^2dt\leq C\mathbb{E}\int^{T\wedge\tau^N_n}_0(1+\|\eta^n(s)\|^2)\left(\int^{d^*_n(s)}_sdt\right)ds\leq C\frac{T}{n}.
\end{eqnarray}
From (\ref{qq-1}), we have
\begin{eqnarray*}
\|v^n(d^*_n(t)^-)-v^n(t)\|^2=2\int^{d^*_n(t)}_t \langle A( v^n(s)-v^n(t)), dv^n(s)\rangle=\sum^2_{i=1}I_i(t),
\end{eqnarray*}
where
\begin{eqnarray*}
I_1(t)&=&-2(1-\varepsilon)\int^{d^*_n(t)}_t\langle A(v^n(s)-v^n(t)), Av^n(s)\rangle ds,\\
I_2(t)&=&-2\int^{d^*_n(t)}_t\langle A(v^n(s)-v^n(t)), B(v^n(s),v^n(s))\rangle ds,
\end{eqnarray*}

Using Lemma \ref{lem-8} and the Young inequality, we have
\begin{eqnarray*}
\Big|\mathbb{E}\int^{T\wedge\tau^N_n}_0 I_1(t)dt\Big|
&=&\Big|(1-\varepsilon)\mathbb{E}\int^{T\wedge\tau^N_n}_0\int^{d^*_n(t)}_t(-2|Av^n(s)|^2+2|Av^n(s)||Av^n(t)|)dsdt\Big|\\
&\leq& \Big|(1-\varepsilon)\mathbb{E}\int^{T\wedge\tau^N_n}_0\int^{d^*_n(t)}_t(-2|Av^n(s)|^2+2|Av^n(s)|^2+\frac{1}{2}|Av^n(t)|^2)dsdt\Big|\\
&\leq&  \frac{1-\varepsilon}{2}\mathbb{E}\int^{T\wedge\tau^N_n}_0|Av^n(t)|^2\left(\int^{d^*_n(t)}_t ds\right)dt\leq \frac{C\tilde{K}(N)T}{n}.
\end{eqnarray*}
Using H\"{o}lder inequality, interpolation inequality and the Young inequality, we obtain
\begin{eqnarray*}
|\langle Av^n(s), B(v^n(s),v^n(s))\rangle|
&\leq& \frac{1}{4}|Av^n(s)|^2+C(|v^n|^2\|v^n\|^2+|r^n|\|r^n\|)\|v^n\|^2,\\
|\langle Av^n(t), B(v^n(s),v^n(s))\rangle|
 &\leq& \frac{1}{4}|Av^n(t)|^2+\frac{1}{4}|Av^n(s)|^2+C(|v^n|^2\|v^n\|^2+|r^n|\|r^n\|)\|v^n\|^2,
\end{eqnarray*}

Hence, by Lemmas \ref{lem-3}, \ref{lem-6}, \ref{lem-8} and \ref{lem-9}, we deduce that
\begin{eqnarray*}
\Big|\mathbb{E}\int^{{T\wedge \tau^N_n}}_0I_2(t)dt\Big|
&\leq&  \frac{1}{2}\mathbb{E}\int^{T\wedge \tau^N_n}_0|Av^n(s)|^2\left(\int^{s}_{d_n(s)} dt\right)ds+\frac{1}{4}\Big|\mathbb{E}\int^{T\wedge \tau^N_n}_0|Av^n(t)|^2\left(\int^{d^*_n(t)}_tds\right)dt\Big|\\
&&\ +C\Big|\mathbb{E}\int^{T\wedge \tau^N_n}_0(|v^n|^2\|v^n\|^4+|r^n|^2\|r^n\|^2+\|v^n\|^4)\left(\int^{s}_{d_n(s)} dt\right)ds\Big|\\
&\leq&\frac{C\tilde{K}(N)T}{n}.
\end{eqnarray*}
Therefore, based on the above, we conclude that (\ref{qq-56}) holds when $\varepsilon =0$.

\textbf{Case 2:}  $\varepsilon\in (0,1)$. We have
\begin{eqnarray*}
\mathbb{E}\int^{T\wedge \tau^N_n}_0\|\eta^n(t)-v^n(t)\|^2dt
&=&-2(1-\varepsilon)\mathbb{E}\int^{T\wedge \tau^N_n}_0\int^{d^*_n(t)}_t\langle Av^n(s), A(\eta^n(s)-v^n(s))\rangle dsdt\\
&&-2\mathbb{E}\int^{T\wedge \tau^N_n}_0\int^{d^*_n(t)}_t\langle B(v^n(s), v^n(s)), A(\eta^n(s)-v^n(s))\rangle dsdt\\
&&-2\varepsilon\mathbb{E}\int^{T\wedge \tau^N_n}_0\int^{t}_{d_n(t)}\langle A(\eta^n(s)), A(\eta^n(s)-v^n(s))\rangle dsdt\\
&&+\mathbb{E}\int^{T\wedge \tau^N_n}_0\int^{t}_{d_n(t)}\|\psi(s,\eta^n(s))\|^2_{\mathcal{L}_2(U;V)}dsdt\\
&:=& J_1+J_2+J_3+J_4.
\end{eqnarray*}
Note that
\begin{eqnarray}\notag
\langle Au, A(y-u)\rangle&=&\langle A(u-y), A(y-u)\rangle+\langle Ay, A(y-u)\rangle\\
\label{qq-61}
&&=-|A(y-u)|^2+\langle Ay, A(y-u)\rangle\leq \langle Ay, A(y-u)\rangle;\\
\notag
2\langle Au, A(y-u)\rangle&\leq& \langle Au, A(y-u)\rangle+\langle Ay, A(y-u)\rangle\\
\label{qq-62}
&=&\langle A(y+u), A(y-u)\rangle\leq\langle Ay, Ay\rangle.
\end{eqnarray}
By (\ref{qq-62}) and Fubini theorem, we have
\begin{eqnarray}\notag
 J_1&\leq& (1-\varepsilon)\mathbb{E}\int^{T\wedge \tau^N_n}_0|A\eta^n(s)|^2\left(\int^{s}_{d_n(s)}dt\right)ds\\
\label{qq-67}
&\leq& \frac{(1-\varepsilon)T}{n}\mathbb{E}\int^{T\wedge \tau^N_n}_0|A\eta^n(s)|^2ds\leq\frac{C\tilde{K}(N)(1-\varepsilon)T}{n}.
\end{eqnarray}
Similar to the above, we have
\begin{eqnarray*}
 J_2\leq C\mathbb{E}\int^{T\wedge \tau^N_n}_0\int^{d^*_n(t)}_t(|Av^n(s)|^2+|A\eta^n(s)|^2+|v^n|^2\|v^n\|^4
+|r^n|^2\|r^n\|^2+\|v^n\|^4)dsdt.
\end{eqnarray*}
With the help of Fubini theorem, Lemma \ref{lem-4}, we get
\begin{eqnarray}\notag
 J_2&\leq& C\mathbb{E}\int^{T\wedge \tau^N_n}_0(|Av^n(s)|^2+|A\eta^n(s)|^2+|v^n|^2\|v^n\|^4
+|r^n|^2\|r^n\|^2+\|v^n\|^4)\left(\int^{s}_{d_n(s)}dt\right) ds\\
\label{qq-66}
&\leq& \frac{C\tilde{K}(N)T}{n}.
\end{eqnarray}
From Lemma \ref{lem-8}, it's easy to obtain
\begin{eqnarray}\notag
 J_3&\leq&  C\varepsilon\mathbb{E}\int^{T\wedge \tau^N_n}_0\int^{t}_{d_n(t)}(|Av^n(s)|^2+|A\eta^n(s)|^2)dsdt\\ \notag
&\leq& C\mathbb{E}\int^{T\wedge \tau^N_n}_0(|Av^n(s)|^2+|A\eta^n(s)|^2)\left(\int^{s}_{d_n(s)}dt\right) ds\\
\label{qq-65}
&\leq& \frac{C\tilde{K}(N)T}{n}.
\end{eqnarray}
We deduce from \textbf{Hypothesis C}, Lemma \ref{lem-3} and Lemma \ref{lem-8} that
\begin{eqnarray}\notag
 J_4&\leq&  \mathbb{E}\int^{T\wedge \tau^N_n}_0\int^{t}_{d_n(t)}(R_0+R_1\|\eta^n(s)\|^2+\varepsilon R_2|A\eta^n(s)|^2)dsdt\\ \notag
&\leq& C \frac{R_0 T}{n}+C\mathbb{E}\int^{T\wedge \tau^N_n}_0(R_1\|\eta^n(s)\|^2+\varepsilon R_2|A\eta^n(s)|^2)\left(\int^{s}_{d_n(s)}dt\right) ds\\
\label{qq-68}
&\leq& \frac{C\tilde{K}(N)T}{n}.
\end{eqnarray}
Combining (\ref{qq-67})-(\ref{qq-68}), we complete the proof of (\ref{qq-56}) when $\varepsilon\in (0,1)$.

\end{proof}

\subsection{Auxiliary process}
For technical reasons, consider an auxiliary process $Z^n(t), t\in [0,T]$ defined by
\begin{eqnarray*}
Z^n(t)=v_0-\int^t_0F_\varepsilon(s,v^n(s))ds-\varepsilon\int^{d_n(t)}_0 A\eta^n(s)ds+\int^t_0\psi(s,\eta^n(s))dW(s).
\end{eqnarray*}
When $\varepsilon=0$, we have
\begin{eqnarray*}
Z^n(t_k)=\eta^n(t^-_k)=v^n(t^+_k)\quad {\rm{for}}\quad k=0,1,\cdot\cdot\cdot,n.
\end{eqnarray*}

The following lemma gives an estimate of the difference between $Z^n$ and $v^n$ in different topologies.
\begin{lemma}\label{lem-7}
Let $v_0\in V$ be $\mathcal{F}_0-$measurable random variable. Fix $\varepsilon\in [0,1)$.
\begin{description}
  \item[(i)]  Suppose that \textbf{Hypotheses A, B} hold with $K_2<\frac{2}{2p-1},\ K_4<2$ and $L_2<\frac{2}{3}, \ R_2<\frac{2}{2p-1}$. Then there exists a positive constant $C:=C(T,\varepsilon, \mathbb{E}|v_0|^{2p})$ such that for every integer $n\geq 1$,
\begin{eqnarray}\label{qq-35}
\sup_{t\in[0,T\wedge \tau^N_n]}\mathbb{E}|Z^n(t)-v^n(t)|^{2p}\leq \frac{C\tilde{K}(N)}{n^p}.
\end{eqnarray}
Moreover, if $L_2=0$, we obtain
\begin{eqnarray}\label{qq-80}
\sup_{t\in[0,T]}\mathbb{E}|\partial_z(Z^n(t)-v^n(t))|^{2p}\leq \frac{C}{n^p}.
\end{eqnarray}
  \item[(ii)] Assume that \textbf{Hypothesis A, B, C} hold with $K_2<\frac{2}{3},\ K_4<2$ and $L_2<\frac{2}{3},\ R_2<2$. Then there exists a positive constant $C:=C(T,\varepsilon,\mathbb{E}\|v_0\|^{2p})$ such that for every integer $n\geq 1$,
\begin{eqnarray*}\label{qq-36}
\mathbb{E}\int^{T\wedge \tau^N_n}_0\|Z^n(t)-v^n(t)\|^{2}dt\leq \frac{C\tilde{K}(N)}{n}.
\end{eqnarray*}
Moreover, if $L_2<\frac{2}{2p-1}$ and $ R_2=0$, we have
\begin{eqnarray*}\label{qq-41}
\sup_{t\in[0,T\wedge \tau^N_n]}\mathbb{E}\|Z^n(t)-v^n(t)\|^{2p}\leq \frac{C\tilde{K}(N)}{n^p}.
\end{eqnarray*}
\end{description}
\end{lemma}
\begin{proof}
 For $t\in[0,T\wedge \tau^N_n]$, we have
\[
Z^n(t)-v^n(t)=\int^t_{d_n(t)}\psi(s,\eta^n(s))dW(s).
\]
\textbf{(i)}\ Applying the Burkholder-Davies-Gundy inequality, \textbf{Hypothesis A}, Lemma \ref{lem-4} and Lemma \ref{lem-9}, we obtain
\begin{eqnarray*}
\mathbb{E}|Z^n(t)-v^n(t)|^{2p}&\leq& C_p\mathbb{E}|\int^t_{d_n(t)}\|\psi(s,\eta^n(s))\|^2_{\mathcal{L}_2(U;H)}ds|^p\\
&\leq& C_p(\frac{T}{n})^{p-1}\mathbb{E}|\int^t_{d_n(t)}|K_0+K_1|\eta^n(s)|^2+\varepsilon K_2\|\eta^n(s)\|^2|^pds\\
&\leq& \frac{C_p(T)}{n^p}\left(K^p_0+K^p_1\sup_{t\in[0,T\wedge \tau^N_n]}\mathbb{E}|\eta^n(t)|^{2p}+\varepsilon^p K^p_2\sup_{t\in[0,T\wedge \tau^N_n]}\mathbb{E}\|\eta^n(t)\|^{2p}\right)\\
&\leq& \frac{C_p(T)\tilde{K}(N)}{n^p}.
\end{eqnarray*}
Note that
\[
\partial_z(Z^n(t)-v^n(t))=\int^t_{d_n(t)}\partial_z\psi(s,\eta^n(s))dW(s).
\]
When $L_2=0$, using \textbf{Hypothesis B} and Lemma \ref{lem-6}, we deduce that
\begin{eqnarray*}
\mathbb{E}|\partial_z(Z^n(t)-v^n(t))|^{2p}&\leq& C_p\mathbb{E}|\int^t_{d_n(t)}\|\partial_z\psi(s,\eta^n(s))\|^2_{\mathcal{L}_2(U;H)}ds|^p\\
&\leq& C_p(\frac{T}{n})^{p-1}\mathbb{E}|\int^t_{d_n(t)}|L_0+L_1|\partial_z\eta^n(s)|^2|^pds\\
&\leq& \frac{C_p(T)}{n^p}(L^p_0+L^p_1\sup_{t\in[0,T]}\mathbb{E}|\partial_z\eta^n(t)|^{2p})\leq \frac{C_p(T)}{n^p}.
\end{eqnarray*}

\textbf{(ii)} \  With the aid of \textbf{Hypothesis C}, the Burkholder-Davies-Gundy inequality, the Fubini's theorem and Lemmas \ref{lem-3}, \ref{lem-8}, we get
\begin{eqnarray*}
\mathbb{E}\int^{T\wedge \tau^N_n}_0\|Z^n(t)-v^n(t)\|^{2}dt
&\leq&\int^{T\wedge \tau^N_n}_0\mathbb{E}\int^t_{d_n(t)}\|\psi(s,\eta^n(s))\|^2_{\mathcal{L}_2(U;V)}dsdt\\
&\leq&\mathbb{E}\int^T_0(R_0+R_1\|\eta^n(s)\|^2+\varepsilon R_2|A\eta^n(s)|^2)\Big(\int^{d^*_n(s)}_sdt\Big)ds\\
&\leq&\frac{T}{n}[R_0 T+R_1\mathbb{E}\int^{T\wedge \tau^N_n}_0\|\eta^n(s)\|^2ds+\varepsilon R_2\mathbb{E}\int^{T\wedge \tau^N_n}_0|A\eta^n(s)|^2ds]\\
&\leq&\frac{C(T)\tilde{K}(N)}{n}.
\end{eqnarray*}
If $R_2=0$, by \textbf{Hypothesis C} and Lemma \ref{lem-9}, it gives
\begin{eqnarray*}
\mathbb{E}\|Z^n(t)-v^n(t)\|^{2p}&\leq& C_p\mathbb{E}|\int^t_{d_n(t)}\|\psi(s,\eta^n(s))\|^2_{\mathcal{L}_2(U;V)}ds|^p\\
&\leq& C_p(\frac{T}{n})^{p-1}\mathbb{E}|\int^t_{d_n(t)}|R_0+R_1\|\eta^n(s)\|^2|^pds\\
&\leq& \frac{C_p(T)}{n^p}(R^p_0+R^p_1\sup_{t\in[0,{T\wedge \tau^N_n}]}\mathbb{E}\|\eta^n(t)\|^{2p})\\
&\leq&\frac{C_p(T)\tilde{K}(N)}{n^p}.
\end{eqnarray*}
\end{proof}

From Propositions \ref{prp-1}, \ref{prp-2}  and Lemma \ref{lem-7}, we deduce that
\begin{cor}\label{cor-2}
There exists a positive constant $C:=C(T,\varepsilon)$ such that for every integer $n\geq 1$,
\begin{eqnarray*}
\mathbb{E}\int^{T\wedge \tau^N_n}_0|Z^n(t)-\eta^n(t)|^{2}dt&\leq& \frac{C\tilde{K}(N)}{n},\\
\mathbb{E}\int^{T\wedge \tau^N_n}_0\|Z^n(t)-\eta^n(t)\|^{2}dt&\leq& \frac{C\tilde{K}(N)}{n}.
\end{eqnarray*}
\end{cor}

\section{Speed of convergence}
In this section, we devote to prove Theorem \ref{thm-4}.

For the strong solution $ v$ of (\ref{equ-7}), $v^n$ of (\ref{qq-1}), $r^n$ of (\ref{qq-18}) and some $M>0$, define the stopping time
\[
\varsigma^M_n=\inf\Big\{t\in[0,T]: \int^t_0(\|v(s)\|+\|v^n(s)\|^2+|r^n|^4)ds> M\Big\}.
\]
Set $\tau:=\varsigma^M_n\wedge \tau^N_n$, where $\tau^N_n$ is defined by (\ref{eq-74}).

The following proposition states that the strong speed of convergence of $Z^n$ to $v$ (resp. $v^n$ and $\eta^n$ to $v$) in $L^{\infty}([0,T\wedge \tau];H)$ (resp. $L^{\infty}([0,T\wedge \tau];V)$ ) is $\frac{1}{2}$.
\begin{prp}\label{prp-8}
Let $v_0\in V$ be $\mathcal{F}_0$ measurable random variable. For any $\varepsilon\in [0,1)$, assume \textbf{Hypotheses A, B, C} hold with $K_2<\frac{2}{147}$, $L_2=R_2=0$ and $\varepsilon K_4$ strictly smaller than $2(1-\varepsilon)$, then there exists positive constant  $C(T)$ such that for every $M>0$ and $n\in\mathbb{N}$, we have
\begin{eqnarray}\label{qq-39}
\mathbb{E}\left(\sup_{t\in[0,T\wedge\tau]}|Z^n(t)-v(t)|^2+\int^{T\wedge\tau}_0\Big(\|v^n(t)-v(t)\|^2+\|\eta^n(t)-v(t)\|^2\Big)dt\right)\leq \frac{K(T,M,N)}{n},
\end{eqnarray}
where
\[
K(T,M,N)=C(T)\tilde{K}(N)\exp\{C(T){\rm{e}}^{C(b_0)M}\},\quad \tilde{K}(N)=\frac{1}{N}e^{C(T)N},\quad C(b_0)\ {\rm{is\ a \ positive\ constant}}.
\]
\end{prp}
\begin{proof}
Fix $M>0$ and $n\geq1$. Then for any $t\in[0,T]$, we have
\begin{eqnarray*}\notag
Z^n(t\wedge \tau)-v(t\wedge \tau)&=&-\int^{t\wedge \tau}_0[F_{\varepsilon}(s, v^n(s))-F(s, v(s))]ds-\varepsilon\int^{d_n(t\wedge \tau)}_0 A\eta^n(s)ds\\
&&\ +\int^{t\wedge \tau}_0[\psi(s, \eta^n(s))-\psi(s, v(s))]dW(s).
\end{eqnarray*}
Applying It\^{o} formula to $|Z^n(t\wedge \tau)-v(t\wedge \tau)|^2$, we get
\begin{eqnarray*}
|Z^n(t\wedge \tau)-v(t\wedge \tau)|^2=\sum^5_{i=1}J_i(t),
\end{eqnarray*}
where
\begin{eqnarray*}
J_1(t)&=&-2\int^{t\wedge \tau}_0\langle F_{\varepsilon}(s, v^n(s))-F_{\varepsilon}(s, v(s)), Z^n(s)-v(s) \rangle ds,\\
J_2(t)&=&-2\varepsilon\int^{d_n(t\wedge \tau)}_0\langle A\eta^n(s)-Av(s), Z^n(s)-v(s)\rangle ds,\\
J_3(t)&=&-2\varepsilon\int^{t\wedge \tau}_{d_n(t\wedge \tau)}\langle Av(s), Z^n(s)-v(s)\rangle ds,\\
J_4(t)&=&\int^{t\wedge \tau}_0 \|\psi(s,\eta^n(s))-\psi(s,v(s))\|^2_{\mathcal{L}_2(U;H)}ds,\\
J_5(t)&=&2\int^{t\wedge \tau}_0\langle [\psi(s,\eta^n(s))-\psi(s,v(s))]dW(s), Z^n(s)-v(s) \rangle.
\end{eqnarray*}
Using (\ref{ee-2}), $J_1(t)$ can be rewritten as
\begin{eqnarray*}
J_1(t)&=& -2(1-\varepsilon)\int^{t\wedge \tau}_0\langle A v^n(s)-Av(s), v^n(s)-v(s) \rangle ds\\
&&\ -2(1-\varepsilon)\int^{t\wedge \tau}_0\langle A v^n(s)-Av(s), Z^n(s)-v^n(s) \rangle ds\\
&&\ -2\int^{t\wedge \tau}_0\langle B(v^n(s)-v(s), v^n(s)), v^n(s)-v(s) \rangle ds\\
&&\
-2\int^{t\wedge \tau}_0\langle [B(v^n(s)-v(s), v^n(s))+B( v(s),v^n(s)-v(s))], Z^n(s)-v^n(s) \rangle ds\\
&:=& J_{1,1}(t)+J_{1,2}(t)+J_{1,3}(t)+J_{1,4}(t).
\end{eqnarray*}
Referring to Page 21-23 in \cite{B-B}, the following estimates hold:
\begin{eqnarray*}
J_{1,1}(t)&=&-2(1-\varepsilon)\int^{t\wedge \tau}_0\|v^n(s)-v(s)\|^2ds,\\
J_{1,2}(t)&\leq& b_0(1-\varepsilon)\int^{t\wedge \tau}_0\|v^n(s)-v(s)\|^2ds+\frac{1-\varepsilon}{b_0}\int^{t\wedge \tau}_0\|Z^n(s)-v^n(s)\|^2ds.
\end{eqnarray*}
Using (\ref{ee-2}) and the Young inequality, we obtain
\begin{eqnarray*}\notag
J_{1,3}(t)&\leq& 2\int^{t\wedge \tau}_0|\langle B(v^n(s)-v(s), v^n(s)), v^n-v(s) \rangle| ds\\
&\leq&2C\int^{t\wedge \tau}_0(\|v^n(s)\||v^n(s)-v(s)|\|v^n(s)-v(s)\|+|\partial_z v^n|\|v^n(s)-v(s)\|^{\frac{3}{2}}|v^n(s)-v(s)|^{\frac{1}{2}})ds\\
&\leq& b_0\int^{t\wedge \tau}_0\|v^n(s)-v(s)\|^2ds+C(b_0)\int^{t\wedge \tau}_0(\|v^n(s)\|^2+|\partial_z v^n|^4)|Z^n(s)-v(s)|^2ds\\
&&\ +C(b_0)\int^{t\wedge \tau}_0(\|v^n(s)\|^2+|\partial_z v^n|^4)|Z^n(s)-v^n(s)|^2ds.
\end{eqnarray*}
$J_{1,4}(t)$ can be rewritten as
\begin{eqnarray*}
J_{1,4}(t)&=&-2\int^{t\wedge \tau}_0\langle B(v^n(s)-v(s), v^n(s)), Z^n(s)-v^n(s) \rangle ds\\
&&\ -2\int^{t\wedge \tau}_0\langle B( v(s),v^n(s)-v(s)), Z^n(s)-v^n(s) \rangle ds\\
&:=& \tilde{J}_1(t)+\tilde{J}_2(t).
\end{eqnarray*}
Using (\ref{ee-3}) and the Young inequality, we get
\begin{eqnarray*}
\tilde{J}_1(t)&\leq&2\int^{t\wedge \tau}_0(\|v^n(s)\|^{\frac{3}{4}}|Z^n(s)-v^n(s)|^{\frac{1}{2}}\|Z^n(s)-v^n(s)\|^{\frac{1}{2}}\|v^n(s)\|^{\frac{1}{4}}|v^n(s)-v(s)|^{\frac{1}{2}}\|v^n(s)-v(s)\|^{\frac{1}{2}}\\
&&\ +|\partial_z v^n|\|v^n(s)-v(s)\||Z^n(s)-v^n(s)|^{\frac{1}{2}}\|Z^n(s)-v^n(s)\|^{\frac{1}{2}})ds\\
&\leq& C\int^{t\wedge \tau}_0(\|v^n(s)\|^{\frac{3}{2}}|Z^n(s)-v^n(s)|\|Z^n(s)-v^n(s)\|+\|v^n(s)\|^{\frac{1}{2}}|v^n(s)-v(s)|\|v^n(s)-v(s)\|\\
&&\ +\frac{b_0}{2}\|v^n(s)-v(s)\|^2+C(b_0)|\partial_z v^n|^4|Z^n(s)-v^n(s)|^2+C(b_0)\|Z^n(s)-v^n(s)\|)ds\\
&\leq& b_0\int^{t\wedge \tau}_0\|v^n(s)-v(s)\|^2ds+C(b_0)\int^{t\wedge \tau}_0\|v^n(s)\||Z^n(s)-v(s)|^2ds\\
&&\ +C(b_0)\int^{t\wedge \tau}_0(1+\|v^n(s)\|^3+|\partial_z v^n|^4)|Z^n(s)-v^n(s)|^2ds\\
&&\
+C(b_0)\int^{t\wedge \tau}_0\|Z^n(s)-v^n(s)\|^2ds.
\end{eqnarray*}
We deduce from (\ref{ee-3}) and (\ref{ee-2}) that
\begin{eqnarray*}
\tilde{J}_2(t)&\leq&2\int^{t\wedge \tau}_0|\langle B( v(s),Z^n(s)-v^n(s)), v^n(s)-v(s) \rangle |ds\\
&\leq&2\int^{t\wedge \tau}_0\|Z^n(s)-v^n(s)\||v(s)|^{\frac{1}{2}}\|v(s)\|^{\frac{1}{2}}|v^n(s)-v(s)|^{\frac{1}{2}}\|v^n(s)-v(s)\|^{\frac{1}{2}}ds\\
&&\ +2\int^{t\wedge \tau}_0|\partial_z (Z^n(s)-v^n(s))|\|v(s)\||v^n(s)-v(s)|^{\frac{1}{2}}\|v^n(s)-v(s)\|^{\frac{1}{2}}ds\\
&:=& \tilde{J}_{2,1}+\tilde{J}_{2,2}.
\end{eqnarray*}
Applying the Cauchy-Schwarz inequality and the Young inequality, we obtain
\begin{eqnarray*}
\tilde{J}_{2,1}(t)&\leq&\int^{t\wedge \tau}_0(b_0\|v^n(s)-v(s)\|^2+C|v(s)|\|v(s)\|^{\frac{1}{2}}\|Z^n(s)-v^n(s)\|^2+C(b_0)\|v(s)\||v^n(s)-v(s)|^2)ds\\
&\leq&b_0\int^{t\wedge \tau}_0\|v^n(s)-v(s)\|^2ds+C(b_0)\int^{t\wedge \tau}_0\|v(s)\||Z^n(s)-v(s)|^2ds\\
&&\ + C(b_0)\int^{t\wedge \tau}_0\|v(s)\||Z^n(s)-v^n(s)|^2ds+C\int^{t\wedge \tau}_0(|v(s)|^2+\|v(s)\|)\|Z^n(s)-v^n(s)\|^2ds.
\end{eqnarray*}
By the H\"{o}lder inequality and the Young inequality, we deduce that
\begin{eqnarray*}
\tilde{J}_{2,2}(t)
&\leq& \int^{t\wedge \tau}_0(|\partial_z (Z^n-v^n)|^2\|v(s)\|^{\frac{3}{2}}+\|v(s)\|^{\frac{1}{2}}|v^n(s)-v(s)|\|v^n(s)-v(s)\|)ds\\
&\leq& \int^{t\wedge \tau}_0(b_0\|v^n(s)-v(s)\|^2+|\partial_z (Z^n-v^n)|^2\|v(s)\|^{\frac{3}{2}}+C(b_0)\|v(s)\||v^n(s)-v(s)|^2)ds\\
&\leq& b_0\int^{t\wedge \tau}_0\|v^n(s)-v(s)\|^2ds+C(b_0)\int^{t\wedge \tau}_0(1+\|v(s)\|)|Z^n(s)-v(s)|^2ds\\
&&\ +C(b_0)\int^{t\wedge \tau}_0(1+\|v(s)\|)|Z^n(s)-v^n(s)|^2ds +C\int^{t\wedge \tau}_0|\partial_z (Z^n-v^n)|^2\|v(s)\|^{\frac{3}{2}}ds.
\end{eqnarray*}
Hence, we have
\begin{eqnarray*}\notag
J_{1,4}(t)&\leq&3 b_0\int^{t\wedge \tau}_0\|v^n(s)-v(s)\|^2ds+C(b_0)\int^{t\wedge \tau}_0(1+\|v(s)\|+\|v^n(s)\|)|Z^n(s)-v(s)|^2ds \\ \notag
&&\ +C(b_0)\int^{t\wedge \tau}_0(1+\|v^n(s)\|^3+|\partial_z v^n|^4+\|v(s)\|)|Z^n(s)-v^n(s)|^2ds\\
&&\ +C(b_0)\int^{t\wedge \tau}_0(1+|v(s)|^2+\|v(s)\|)\|Z^n(s)-v^n(s)\|^2ds\\
&&+C\int^{t\wedge \tau}_0\|v(s)\|^{\frac{3}{2}}|\partial_z (Z^n-v^n)|^2ds.
\end{eqnarray*}
Replacing $v$ by $v^n$, and using the Cauchy-Schwarz inequality and the Young inequality, we obtain
\begin{eqnarray*}\notag
J_2(t)&\leq& -2\varepsilon\int^{d_n(t\wedge \tau)}_0\| v^n(s)-v(s)\|^2 ds+2\varepsilon\int^{d_n(t\wedge \tau)}_0\| \eta^n(s)-v^n(s)\|\|Z^n(s)-v^n(s)\| ds\\ \notag
&& + 2\varepsilon\int^{d_n(t\wedge \tau)}_0\| v^n(s)-v(s)\|(\| \eta^n(s)-v^n(s)\|+\|Z^n(s)-v^n(s)\|) ds\\
&\leq& C(\varepsilon)\int^{d_n(t\wedge \tau)}_0\| \eta^n(s)-v^n(s)\|^2ds
 +C(\varepsilon)\int^{d_n(t\wedge \tau)}_0\| Z^n(s)-v^n(s)\|^2ds.
\end{eqnarray*}
We deduce from the Cauchy-Schwarz inequality and the Young inequality that
\begin{eqnarray*}\notag
J_3(t)&\leq&2\varepsilon\int^{t\wedge \tau}_{d_n(t\wedge \tau)}\|v(s)\|( \|Z^n(s)-v^n(s)\|+\|v^n(s)-v(s)\|) ds\\ \notag
&\leq&b_0\varepsilon\int^{t\wedge \tau}_{d_n(t\wedge \tau)}\|v^n(s)-v(s)\|^2ds+C(\varepsilon)\int^{t\wedge \tau}_{d_n(t\wedge \tau)}\|v^n(s)-v(s)\|\|Z^n(s)-v^n(s)\|ds\\ \notag
&&+C(\varepsilon)\int^{t\wedge \tau}_{d_n(t\wedge \tau)}(\|v^n(s)\|\|Z^n(s)-v^n(s)\|+\|v^n(s)\|\|v^n(s)-v(s)\|)ds\\ \notag
&\leq&2b_0\varepsilon\int^{t\wedge \tau}_{d_n(t\wedge \tau)}\|v^n(s)-v(s)\|^2ds+C(\varepsilon)\int^{t\wedge \tau}_{d_n(t\wedge \tau)}\|Z^n(s)-v^n(s)\|^2ds\\
&&+C(\varepsilon)\frac{T}{n}\sup_{{d_n(t\wedge \tau)}\leq s\leq{{t\wedge \tau}}}\|v^n(s)\|^2.
\end{eqnarray*}
Using \textbf{Hypothesis A}, we obtain
\begin{eqnarray*}\notag
J_4(t)&\leq&\int^{t\wedge \tau}_0 (K_3|\eta^n(s)-v(s)|^2+\varepsilon K_4\|\eta^n(s)-v(s)\|^2 )ds\\ \notag
&\leq& 2K_3 \int^{t\wedge \tau}_0|Z^n(s)-v(s)|^2ds+\varepsilon K_4 b_0\int^{t\wedge \tau}_0\|v^n(s)-v(s)\|^2ds\\
&&\ +2K_3 \int^{t\wedge \tau}_0|\eta^n(s)-Z^n(s)|^2ds+\varepsilon C\int^{t\wedge \tau}_0\|\eta^n(s)-v^n(s)\|^2ds.
\end{eqnarray*}
Choosing $b_0>0$ satisfies
\[
2(1-\varepsilon)-b_0(1-\varepsilon)-3b_0-2b_0\varepsilon-\varepsilon K_4 b_0>\alpha>0,\ {\rm{for \ some\  }} \alpha>0.
\]
For $t\in [0,T]$, define
\[
X(t)=\sup_{s\in[0, t\wedge \tau]}|Z^n(s)-v(s)|^2, \quad  Y(t)=\int^{t\wedge \tau}_0\|v^n(s)-v(s)\|^2ds.
\]
Then,
\begin{eqnarray*}
X(t)+\alpha Y(t)\leq \int^{t\wedge \tau}_0 \Theta_1(s)X(s)ds+\Theta_2(t),
\end{eqnarray*}
where the processes are defined as follows:
\begin{eqnarray*}\notag
\Theta_1(s)&=&C(b_0)(1+\|v(s)\|+\|v^n(s)\|^2+|r^n|^4),\\ \notag
\Theta_2(t)&=&\sup_{s\in[0, t\wedge \tau]}|J_5(s)|+I(t),\\  \notag
I(t)&=&C(b_0)\int^{t\wedge \tau}_0(1+\|v^n(s)\|^3+|\partial_z v^n|^4+\|v(s)\|)|Z^n(s)-v^n(s)|^2ds\\ \notag
&&\
+C(b_0)\int^{t\wedge \tau}_0(1+|v|^2+\|v\|)\|Z^n(s)-v^n(s)\|^2ds  +C\int^{t\wedge \tau}_0|\partial_z (Z^n-v^n)|^2\|v(s)\|^{\frac{3}{2}}ds\\ \notag
&&\
+C(\varepsilon)\frac{T}{n}\sup_{{t\wedge \tau}\leq s\leq{d_n(t\wedge \tau)}}\|v^n(s)\|^2+2K_3 \int^{t\wedge \tau}_0|\eta^n(s)-Z^n(s)|^2ds\\
&&\ +\varepsilon C\int^{t\wedge \tau}_0\|\eta^n(s)-v^n(s)\|^2ds.
\end{eqnarray*}
The definition of $\tau$ implies that
\[
\int^{\tau}_0\Theta_1(s)ds\leq C(b_0)(T+M):= C_0,\ P-a.s..
\]
By the Burkholder-Davies-Gundy inequality, \textbf{Hypothesis A}, Proposition \ref{prp-2} and Corollary \ref{cor-2}, we obtain
\begin{eqnarray*}\notag
\mathbb{E}\left(\sup_{0\leq s\leq t\wedge \tau}|J_5(s)|\right)&=&C\mathbb{E}(\int^{t\wedge \tau}_0 \|\psi(s,\eta^n(s))-\psi(s,v(s))\|^2_{\mathcal{L}_2(U;H)}| Z^n(s)-v(s)|^2ds)^{\frac{1}{2}}\\ \notag
&\leq& \beta\mathbb{E}(\sup_{0\leq s\leq t\wedge \tau}| Z^n(s)-v(s)|^2 )+C(\beta)\mathbb{E}\int^{t\wedge \tau}_0 (K_3 |\eta^n(s)-v(s)|^2+\varepsilon K_4 \|\eta^n(s)-v(s)\|^2)ds\\ \notag
&\leq& \beta\mathbb{E}(\sup_{0\leq s\leq t\wedge \tau}| Z^n(s)-v(s)|^2 )+C(\beta)K_3\int^{t\wedge \tau}_0\mathbb{E}|Z^n(s)-v(s)|^2ds\\ \notag
&&\ +C(\beta)K_3\int^{t\wedge \tau}_0\mathbb{E}|\eta^n(s)-Z^n(s)|^2ds+C(\beta)\varepsilon K_4\int^{t\wedge \tau}_0\mathbb{E} \|\eta^n(s)-v^n(s)\|^2ds\\ \notag
&&\ + C(\beta)\varepsilon K_4\int^{t\wedge \tau}_0\mathbb{E} \|v^n(s)-v(s)\|^2ds\\
&\leq& \beta\mathbb{E}X(t)+C(\beta)K_3\int^{t}_0\mathbb{E}X(s)ds +C(\beta)\varepsilon K_4\mathbb{E}Y(t)+\frac{C(T)\tilde{K}(N)}{n},
\end{eqnarray*}
where $\beta>0$ will be chosen later.
Using Theorem \ref{thm-3} and Lemmas \ref{lem-3}, \ref{lem-6}, \ref{lem-8}, \ref{lem-9}, we have
\begin{eqnarray*}\notag
\mathbb{E}I(t)&\leq& C(b_0)T^{\frac{1}{2}}\left(\mathbb{E}\sup_{s\in[0,t\wedge \tau ]}|Z^n(s)-v^n(s)|^4\right)^{\frac{1}{2}}
\left[\mathbb{E}\int^{t\wedge \tau}_0(1+\|v^n(s)\|^6+|r^n(s)|^8+\|v(s)\|^2)ds\right]^{\frac{1}{2}}\\
&&\ +C(b_0)T^{\frac{1}{2}}\left(\mathbb{E}\sup_{s\in[0,t\wedge \tau ]}\|Z^n(s)-v^n(s)\|^4\right)^{\frac{1}{2}}
\left[\mathbb{E}\int^{t\wedge \tau}_0(1+|v(s)|^4+\|v(s)\|^2)ds\right]^{\frac{1}{2}}\\
&&\ +C(T)\left(\mathbb{E}\sup_{s\in[0,t\wedge \tau ]}|\partial_z (Z^n(s)-v^n(s))|^8\right)^{\frac{1}{4}}\left(\mathbb{E}\int^{t\wedge \tau}_0\|v(s)\|^2ds\right)^{\frac{3}{4}}\\ \notag
&&\ +C(\varepsilon)\frac{T}{n}\mathbb{E}\sup_{{d_n(t\wedge \tau)}\leq s\leq{t\wedge \tau}}\|v^n(s)\|^2+2K_3 \mathbb{E}\int^{t\wedge \tau}_0|\eta^n(s)-Z^n(s)|^2ds\\
&&\  +\varepsilon C\mathbb{E}\int^{t\wedge \tau}_0\|\eta^n(s)-v^n(s)\|^2ds\leq \frac{C(T)\tilde{K}(N)}{n}.
\end{eqnarray*}
Choosing $\beta>0$ such that
\[
2\beta (1+C_0 {\rm{e}}^{C(b_0)M})\leq 1,
\]
then suppose $K_4$ is small enough to ensure that
\[
C(\beta)\varepsilon K_4 (1+C_0 {\rm{e}}^{C(b_0)M})\leq \frac{\alpha}{4}.
\]
Then, using similar argument as Lemma 3.9 in \cite{D-M}, we deduce that
\begin{eqnarray*}
X(t)+\frac{\alpha}{2} Y(t)\leq [I(t)+\sup_{0\leq s\leq t\wedge \tau}|M(s)|](1+C_0 {\rm{e}}^{C(b_0)M}).
\end{eqnarray*}
Taking expectation and by estimates of $\mathbb{E}I(t)$, we obtain
\begin{eqnarray*}\notag
\mathbb{E}X(T)+ \frac{\alpha}{4}\mathbb{E} Y(T)
&\leq & 2\frac{C(T)\tilde{K}(N)}{n}(1+C_0 {\rm{e}}^{C(b_0)M})+C(\beta)K_3(1+C_0 {\rm{e}}^{C(b_0)M})\int^t_0 \mathbb{E}X(s)ds.
\end{eqnarray*}
Applying the Gronwall inequality, we have
\begin{eqnarray*}\notag
\mathbb{E}X(T)+ \frac{\alpha}{4}\mathbb{E} Y(T)
&\leq & 2\frac{C(T)\tilde{K}(N)}{n}(1+C_0 {\rm{e}}^{C(b_0)M})\cdot\exp\Big\{C(\beta)K_3T(1+C_0 {\rm{e}}^{C(b_0)M})\Big\}.
\end{eqnarray*}
where $C(T), C_0, C(b_0), C(\beta)$ is independent of $n$.

Finally, with the aid of Proposition \ref{prp-2}, we have
\begin{eqnarray*}
\mathbb{E}\int^{T\wedge\tau}_0\|\eta^n(t)-v(t)\|^2dt
&\leq& \mathbb{E}\int^{T\wedge\tau}_0\|\eta^n(t)-v^n(t)\|^2dt+\mathbb{E}Y(T)\\
&\leq& \frac{C(T)\tilde{K}(N)}{n}\exp\Big\{C(T){\rm{e}}^{C(b_0)M}\Big\}.
\end{eqnarray*}
 We complete the proof.
\end{proof}

\begin{remark}
As explained in the introduction, the index of $\|v\|$ appeared in $I(t)$ has to be strictly less than $2$. Otherwise, $\mathbb{E}I(t)$ can not be controlled because of the lack of uniform $V-$norm estimates of $v$.
\end{remark}

For every $M=M(n)>0, N=N(n)>0$, $t\in [0,T]$ and any integer $n\geq 1$, let
\begin{eqnarray*}
\Omega^{M,N}_n(t)&=&\Big\{\omega\in \Omega: \sup_{i=0,\cdot\cdot\cdot, n-1}\int^{t_{i+1}\wedge t}_{t_i \wedge t}(|v^n(s)|^2\|v^n(s)\|^2+|r^n(s)|\|r^n(s)\|)ds\leq\frac{N}{n}\\
 &&\ and \ \int^t_0(\|v(s)\|+\|v^n(s)\|^2+|r^n|^4)ds\leq M\Big\}.
\end{eqnarray*}
\begin{thm}\label{thm-10} Under the same conditions as Proposition \ref{prp-8}, we have
\begin{eqnarray}\label{ee-50}
\mathbb{E}\left[I_{\Omega^{M,N}_n(t)}\sup_{k=0,\cdot\cdot\cdot, n}\Big(|v^n(t^+_{k})-v(t_k)|+|\eta^n(t^+_{k})-v(t_k)|\Big)\right]&\leq& \frac{K(M,N,T)}{n},\\
\label{ee-51}
\mathbb{E}\left[I_{\Omega^{M,N}_n(t)}\int^t_0(\|v^n(s)-v(s)\|^2+\|\eta^n(s)-v(s)\|^2)ds \right]&\leq& \frac{K(M,N,T)}{n}.
\end{eqnarray}
where $K(M,N,T)=C(T)\tilde{K}(N)\exp\Big\{C(T){\rm{e}}^{C(b_0)M}\Big\}$, $\tilde{K}(N)=\frac{1}{N}e^{C(T)N}$.
\end{thm}
\begin{proof}
On $\Omega^{M,N}_n(t)$, we have $\tau\geq t$. With the aid of Proposition \ref{prp-8}, we deduce that (\ref{ee-51}) holds. For (\ref{ee-50}), by the H\"{o}ler inequality and Lemma \ref{lem-9}, we have
\begin{eqnarray*}
\mathbb{E}\left(\sup_{k=0,\cdot\cdot\cdot, n}|Z^n(t_k\wedge \tau)-\eta^n(t^-_k\wedge \tau)|^2\right)
&=&\mathbb{E}\left(\sup_{k=0,\cdot\cdot\cdot, n}\varepsilon^2|\int^{t_{k+1}\wedge \tau}_{t_k\wedge \tau}A\eta^n(s)ds|^2\right)\\
&\leq& \varepsilon^2 \frac{T}{n}\mathbb{E}\left(\sup_{k=0,\cdot\cdot\cdot, n}\int^{t_{k+1}\wedge \tau}_{t_k\wedge \tau}|A\eta^n(s)|^2ds\right)\\
&\leq& \frac{C(T)\tilde{K}(N)\varepsilon^2}{n}.
\end{eqnarray*}

In view of $Z^n(t_k)=v^n(t^+_k)=\eta^n(t^-_k)$, we deduce from Proposition \ref{prp-8} that
\begin{eqnarray}\label{ee-56}
\mathbb{E}\Big[I_{\Omega^{M,N}_n(T)}\sup_{k=0,\cdot\cdot\cdot, n}(|v^n(t^+_{k})-v(t_k)|^2+|\eta^n(t^-_{k})-v(t_k)|^2)\Big]
\leq \frac{K(M,N,T)}{n}.
\end{eqnarray}
Using \textbf{Hypothesis (A.1)} and Lemma \ref{lem-8}, for $k=0, \cdot\cdot\cdot, n-1$, we get
\begin{eqnarray*}
&&\mathbb{E}\left(\sup_{t\in [t_k\wedge \tau,t_{k+1}\wedge \tau)}|\eta^n(t)-\eta^n(t^+_k)|^2\right)\\
&\leq& \mathbb{E}\left[\int^{t_{k+1}\wedge \tau}_{t_k\wedge \tau}(\frac{\varepsilon}{2}\|\eta^n(t^+_k)\|^2+K_0+K_1|\eta^n(s)|^2+\varepsilon K_2\|\eta^n(s)\|^2)ds\right]\\
&&\ +\mathbb{E}\left[\int^{t_{k+1}\wedge \tau}_{t_k\wedge \tau}|\eta^n(s)-\eta^n(t^+_k)|^2(K_0+K_1|\eta^n(s)|^2+\varepsilon K_2\|\eta^n(s)\|^2)ds\right]^{\frac{1}{2}}\\
&\leq &\frac{1}{2}\mathbb{E}\left(\sup_{t\in [t_k\wedge \tau,t_{k+1}\wedge \tau)}|\eta^n(t)-\eta^n(t^+_k)|^2\right)+C\frac{T}{n}\sup_{s\in[0,T\wedge \tau]}\mathbb{E}\|\eta^n(s)\|^2.
\end{eqnarray*}
So that,
\begin{eqnarray*}
\mathbb{E}\left(\sup_{t\in [t_k\wedge \tau,t_{k+1}\wedge \tau)}|\eta^n(t)-\eta^n(t^+_k)|^2\right)
\leq \frac{C(T)\tilde{K}(N)}{n}.
\end{eqnarray*}

Using \textbf{Hypotheses A, C}, we obtain
\begin{eqnarray*}
&&\mathbb{E}\left(\sup_{t\in [t_k\wedge \tau,t_{k+1}\wedge \tau)}|v(t)-v(t^+_k)|^2\right)\\
&\leq& \mathbb{E}\left(\sup_{t\in [t_k\wedge \tau,t_{k+1}\wedge \tau)}\int^t_{t_k}\Big[|\langle v(s)-v(t^+_k), Av(s) \rangle|+|\langle v(s)-v(t^+_k), B(v(s),v(s))\rangle|\Big]ds\right)\\
&&\ +\mathbb{E}\left(\sup_{t\in [t_k\wedge \tau,t_{k+1}\wedge \tau)}\int^t_{t_k}\Big[(R_0+R_1|v(s)|)|v(s)-v(t^+_k)|+(K_0+K_1|v(s)|^2+\varepsilon K_2 \|v(s)\|^2)\Big]ds\right)\\
&&\ +\mathbb{E}\left(\int^{t_{k+1}}_{t_k}|v(t)-v(t^+_{k})|^2(K_0+K_1|v(s)|^2+\varepsilon K_2\|v(s)\|^2)ds\right)^{\frac{1}{2}}\\
&\leq& \mathbb{E}\left[\sup_{t\in [t_k\wedge \tau,t_{k+1}\wedge \tau)}\int^t_{t_k}\Big(-2\|v(s)\|^2+2\|v(s)\|\|v(t^+_k)\|
+|v(s)|^2_4\|v(s)-v(t^+_{k})\|+|\partial_z(v(s)-v(t^+_{k}))|\|v(s)\|^{\frac{3}{2}}|v(s)|^{\frac{1}{2}}\Big)ds\right]\\
&&\ +\mathbb{E}\left(\sup_{t\in [t_k\wedge \tau,t_{k+1}\wedge \tau)}\int^t_{t_k}\Big[(R_0+R_1|v(s)|)|v(s)-v(t^+_k)|+(K_0+K_1|v(s)|^2+\varepsilon K_2 \|v(s)\|^2)\Big]ds\right)\\
&&\ +\mathbb{E}\left(\int^{t_{k+1}\wedge \tau}_{t_k\wedge \tau}|v(s)-v(t^+_{k})|^2(K_0+K_1|v(s)|^2+\varepsilon K_2\|v(s)\|^2)ds\right)^{\frac{1}{2}}\\
&\leq & \frac{1}{2}\mathbb{E}(\sup_{t\in [t_k\wedge \tau,t_{k+1}\wedge \tau)}|v(t)-v(t^+_k)|^2)+\frac{C}{n}\left(1+\sup_{t\in [0,T\wedge \tau)}\mathbb{E}\|v(t)\|^2+\sup_{t\in [0,T\wedge \tau)}\mathbb{E}(|\partial_z v(t)|^8+|v(t)|^4)\right).
\end{eqnarray*}
Hence, by Lemmas \ref{lem-4}, \ref{lem-6}, \ref{lem-8}, we get
\begin{eqnarray*}
\mathbb{E}\left(\sup_{t\in [t_k\wedge \tau,t_{k+1}\wedge \tau)}|v(t)-v(t^+_k)|^2\right)\leq \frac{C(T)}{n}.
\end{eqnarray*}
Using Lemma \ref{lem-9}, we have
\begin{eqnarray*}
&&\mathbb{E}\left(\sup_{t\in [t_k\wedge \tau,t_{k+1}\wedge \tau)}|v^n(t)-v^n(t^+_k)|^2\right)\\
&\leq& \frac{C}{n}\Big(1+\sup_{t\in [0,T\wedge \tau)}\mathbb{E}(\|v^n\|^4+\|\eta^n\|^4)\Big)\\
&\leq& \frac{C(T)\tilde{K}(N)}{n}.
\end{eqnarray*}

We complete the proof.

\end{proof}

For any $n\geq 1$, define the error term
\begin{eqnarray*}
e_n(T)&=&\sup_{k=0,\cdot\cdot\cdot, n}\Big(|v^n(t^+_{k})-v(t_k)|+|\eta^n(t^-_{k})-v(t_k)|\Big)\\
&&\ +\left(\int^T_0\|v^n(s)-v(s)\|^2ds\right)^{\frac{1}{2}}+\left(\int^T_0\|\eta^n(s)-v(s)\|^2ds\right)^{\frac{1}{2}}.
\end{eqnarray*}
Now, we can prove the strong speed of the convergence in probability.
\begin{flushleft}
\textbf{Proof of Theorem \ref{thm-4}.} \quad Fix a sequence $l(n)\rightarrow \infty$, as $n\rightarrow \infty$. Let $M(n)=\ln( \ln( \ln (l(n))))$, $N(n)=\ln(\ln (l(n)))$, then $M(n)\rightarrow \infty$ and $N(n)\rightarrow \infty$. Note that
\begin{eqnarray*}
&&\mathbb{P}\Big((\Omega^{M(n),N(n)}_n)^c(T)\Big)\\
&\leq & \mathbb{P}\Big(\sup_{i=0,\cdot\cdot\cdot, n-1}\int^{t_{i+1}\wedge T}_{t_i \wedge T}(|v^n(s)|^2\|v^n(s)\|^2+|r^n(s)|\|r^n(s)\|)ds>\frac{N(n)}{n}\Big)\\
&&\ + \mathbb{P}\Big(\int^T_0(\|v(s)\|+\|v^n(s)\|^2+|r^n|^4)ds> M(n)\Big).
\end{eqnarray*}
Clearly, by Lemmas \ref{lem-4}, \ref{lem-6}, we have
\begin{eqnarray*}
&&\mathbb{P}\Big(\sup_{i=0,\cdot\cdot\cdot, n-1}\int^{t_{i+1}\wedge T}_{t_i \wedge T}(|v^n(s)|^2\|v^n(s)\|^2+|r^n(s)|\|r^n(s)\|)ds>\frac{N(n)}{n}\Big)^c\\
&=& \mathbb{P}\Big(\sup_{i=0,\cdot\cdot\cdot, n-1}\int^{t_{i+1}\wedge T}_{t_i \wedge T}(|v^n(s)|^2\|v^n(s)\|^2+|r^n(s)|\|r^n(s)\|)ds\leq\frac{N(n)}{n}\Big)\\
&\leq& \mathbb{P}\Big(\int^{T}_{0}(|v^n(s)|^2\|v^n(s)\|^2+|r^n(s)|\|r^n(s)\|)ds\leq N(n)\Big)\\
&\rightarrow& 1\quad\quad  as \quad n\rightarrow \infty.
\end{eqnarray*}
Using Theorem \ref{thm-3} and Lemmas \ref{lem-3}, \ref{lem-6}, we obtain
\begin{eqnarray*}
\mathbb{P}\Big[\int^T_0(\|v(s)\|+\|v^n(s)\|^2+|r^n|^4)ds> M(n)\Big]\rightarrow0,\quad\quad as \quad n\rightarrow \infty.
\end{eqnarray*}
Hence, when $n\rightarrow \infty$,
\begin{eqnarray}\label{eee-5}
\mathbb{P}\Big((\Omega^{M(n),N(n)}_n)^c(T)\Big)\rightarrow 0.
\end{eqnarray}
Now, we deduce from Chebyshev inequality and Theorem \ref{thm-10}  that
\begin{eqnarray*}
&&\mathbb{P}\Big(e_n(T)\geq \frac{l(n)}{\sqrt{n}}\Big)\\
&\leq& \mathbb{P}\Big((\Omega^{M(n),N(n)}_n)^c(T)\Big)+\frac{n}{l^2(n)}\mathbb{E}\Big(I_{\Omega^{M(n),N(n)}_n(T)}e^2_n(T)\Big)\\
&\leq & \mathbb{P}\Big((\Omega^{M(n),N(n)}_n)^c(T)\Big)+C(T)\frac{1}{N(n)}e^{C(T)N(n)}\frac{n}{l^2(n)}\frac{1}{n}\exp\Big\{C(T)(\ln (\ln (l(n))))^{C(b_0)}\Big\}\\
&\leq & \mathbb{P}\Big((\Omega^{M(n),N(n)}_n)^c(T)\Big)+C(T)\frac{1}{N(n)}\frac{n}{l^2(n)}\frac{1}{n}\exp\Big\{C(T)(\ln (\ln (l(n))))^{C(b_0)\vee1}\Big\}.
\end{eqnarray*}
Since $C(T)(\ln (\ln (l(n))))^{C(b_0)\vee 1}-2\ln (l(n))\rightarrow -\infty$, we have
\begin{eqnarray}\label{eee-4}
C(T)\frac{1}{N(n)}\frac{n}{l^2(n)}\frac{1}{n}\exp\Big\{C(T)(\ln (\ln (l(n))))^{C(b_0)\vee1}\Big\}\rightarrow 0.
\end{eqnarray}
Combing (\ref{eee-5}) and (\ref{eee-4}), we get
\begin{eqnarray*}
\mathbb{P}\Big(e_n(T)\geq \frac{l(n)}{\sqrt{n}}\Big)\rightarrow 0, \quad as\ n\rightarrow \infty.
\end{eqnarray*}
We complete the proof.
\end{flushleft}

$\hfill\blacksquare$


\vskip 0.2cm {\small {\bf  Acknowledgements}\   This work was
 supported by  NSFC (No. 11501195, 11871476, 11801032). China Postdoctoral Science Foundation (No. 2018M641204). Key Laboratory of Random Complex Structures and Data Science, Academy of Mathematics and Systems Science, Chinese Academy of  Sciences (No. 2008DP173182)}. Scientific Research Fund of Hunan Provincial Education Department (No. 17C0953) and the Construct Program of the Key Discipline in Hunan Province.

\def\refname{ References}


\begin{thebibliography}{2}



\bibitem{Adams} R.A. Adams:  \emph{Sobolev Space}. New vork: Academic Press, 1975.
\bibitem{B-B} H. Bessaih, Z. Brze\'{z}niak,  A. Millet: \emph{
Splitting up method for the 2D stochastic Navier-Stokes equations}.
Stoch. Partial Differ. Equ. Anal. Comput. 2, no. 4, (2014), 433-470.

\bibitem{C-T-1} C. Cao, E.S. Titi: \emph{Global well-posedness of the three-dimensional viscous primitive equations of large-scale ocean and atmosphere dynamics}. Ann. of Math. 166, (2007),  245-267.
\bibitem{D-G-T-Z}  A. Debussche, N. Glatt-Holtz, R. Temam, M. Ziane: \emph{Global existence and regularity for the 3D stochastic primitive equations of the ocean and atmosphere with
multiplicative white noise}.  Nonlinearity, 25, (2012), 2093-2118.

\bibitem{RR} Z. Dong, J. Zhai, R. Zhang: \emph{Large devation principles for 3D stochastic primitive equations.} J. Differential Equations 263 (2017), no. 5, 3110-3146.
   \bibitem{D-Z-Z} Z. Dong, J. Zhai, R. Zhang: \emph{Exponential mixing for 3D stochastic primitive equations of the large scale ocean.} Available at	arXiv:1506.08514.

\bibitem {D} P. D\"{o}rsek: \emph{Semigroup splitting and cubature approximations for the stochastic Navier-Stokes Equations.} SIAM J. Numer. Anal. 50, (2012), 729-746.
\bibitem{D-M} J. Duan, A. Millet: \emph{Large deviations for the Boussinesq equations under random influences.}
Stochastic Process. Appl. 119 (2009), no. 6, 2052-2081.



\bibitem{G-S} H. Gao, C. Sun: \emph{Well-posedness and large deviations for the stochastic primitive equations in two space dimensions.} Commun. Math. Sci. Vol.10, No.2, (2012), 575-593.


\bibitem{Ziane} N. Glatt-Holtz, M. Ziane: \emph{The stochastic primitive equations in two space dimensions with multiplicative noise.}
Discrete Contin. Dyn. Syst. Ser. B 10, no. 4, (2008), 801-822.

\bibitem{Guo} B. Guo, D. Huang: \emph{3D stochastic primitive equations of the large-scale ocean: global well-posedness and attractors .} Comm. Math. Phys. 286, no. 2, (2009), 697-723.

\bibitem{G-K-1} I. Gy\"{o}ngy, N.V. Krylov: \emph{On the splitting-up method and stochastic partial differential equations.} Annals of Probability 31-2, (2003), 564-591.
\bibitem{G-K-2} I. Gy\"{o}ngy, N.V. Krylov: \emph{An accelerated splitting-up method for parabolic equations.} SIAM J. Math. Anal. 37, (2005),  1070-1097.

\bibitem{L-T-W-1} J.L. Lions, R. Temam, S. Wang: \emph{New formulations of the primitive equations of atmosphere and applications.} Nonlinearity 5, (1992), 237-288.
 \bibitem{L-T-W-2} J.L. Lions, R. Temam, S. Wang: \emph{On the equations of the large scale ocean.} Nonlinearity 5, (1992),  1007-1053.
\bibitem{L-T-W-3} J.L. Lions, R. Temam, S. Wang: \emph{Models of the coupled atmosphere and ocean.} Computational Mechanics Advance 1, (1993), 1-54.
\bibitem{L-T-W-4} J.L. Lions, R. Temam, S. Wang: \emph{Mathematical theory for the coupled atmosphere-ocean models.} J. Math. Pures Appl. 74, (1995), 105-163.


 \bibitem{JP} J. Pedlosky: \emph{Geophysical Fluid Dynamics}. Springer-Verlag, New vork, 1987.



\bibitem{Z} M. Ziane: \emph{Regularity results of Stokes type system}. App. Anal.58 (1995), no. 3-4, 263-292.


\end{thebibliography}
\end{document}